\documentclass[12pt,reqno,a4paper]{amsart}
\pdfoutput=1
\usepackage[doi=false,isbn=false,style=alphabetic,sorting=nyt,backend=bibtex,maxnames=5,maxalphanames=5,firstinits=true]{biblatex}
\bibliography{tropicalPoints}

\PassOptionsToPackage{hyphens}{url}
\usepackage[colorlinks,plainpages,hypertexnames=false,plainpages=false]{hyperref}
\hypersetup{urlcolor=blue, citecolor=blue, linkcolor=blue}

\usepackage[latin1]{inputenc}
\usepackage{amsmath}
\usepackage{amsthm}
\usepackage{amssymb}
\usepackage{amsfonts}
\usepackage{enumerate}
\usepackage{stmaryrd}
\usepackage{enumitem}
\usepackage{centernot}
\usepackage{todonotes}
\usepackage{mathrsfs}
\usepackage{listings}
\usepackage{multicol}
\usepackage[noend]{algorithmic} % for the algorithmic enviroment

\usepackage{array}
\newcolumntype{L}[1]{>{\raggedright\let\newline\\\arraybackslash\hspace{0pt}}m{#1}}
\newcolumntype{C}[1]{>{\centering\let\newline\\\arraybackslash\hspace{0pt}}m{#1}}
\newcolumntype{R}[1]{>{\raggedleft\let\newline\\\arraybackslash\hspace{0pt}}m{#1}}

\usepackage{tikz}
\usetikzlibrary{arrows}
\usetikzlibrary{decorations.pathreplacing}
\usetikzlibrary{matrix}
\usetikzlibrary{calc}
\usetikzlibrary{shapes}
\usetikzlibrary{patterns}
\usetikzlibrary{fit,backgrounds,scopes}

\newtheoremstyle{theoremstyle}
{10pt}      %  Space above
{5pt}       %  Space below
{\itshape}  %  Body font
{}          %  Indent amount (empty = no indent, \parindent = para indent)
{\bfseries} %  Thm head font
{}         %  Punctuation after thm head
{\newline}      %  Space after thm head: " " = normal interword space;
{}          %  Thm head spec (can be left empty, meaning `normal')

\newtheoremstyle{algorithmstyle}
{10pt}      %  Space above
{5pt}       %  Space below
{}  %  Body font
{}          %  Indent amount (empty = no indent, \parindent = para indent)
{\bfseries} %  Thm head font
{}         %  Punctuation after thm head
{ }      %  Space after thm head: " " = normal interword space;
{}          %  Thm head spec (can be left empty, meaning `normal')

\newtheoremstyle{examplestyle}
{10pt}      %  Space above
{5pt}       %  Space below
{}          %  Body font
{}          %  Indent amount (empty = no indent, \parindent = para indent)
{\bfseries} %  Thm head font
{}         %  Punctuation after thm head
{\newline}      %  Space after thm head: " " = normal interword space;
{}          %  Thm head spec (can be left empty, meaning `normal')

\theoremstyle{theoremstyle}
\newtheorem{theorem}{Theorem}[section]
\newtheorem*{theorem*}{Theorem}
\newtheorem{lemma}[theorem]{Lemma}
\newtheorem{proposition}[theorem]{Proposition}
\newtheorem*{proposition*}{Proposition}
\newtheorem{corollary}[theorem]{Corollary}
\newtheorem*{corollary*}{Corollary}

\theoremstyle{examplestyle}
\newtheorem{example}[theorem]{Example}
\newtheorem{definition}[theorem]{Definition}
\newtheorem{definition*}{Definition}

\newtheorem{remark}[theorem]{Remark}
\newtheorem{remark*}{Remark}
\newtheorem{convention}[theorem]{Convention}

\newtheorem{timings}[theorem]{Timings}
\theoremstyle{algorithmstyle}
\newtheorem{algorithm}[theorem]{Algorithm}

\newcommand{\NN }{\mathbb{N}}
\newcommand{\CC }{\mathbb{C}}
\newcommand{\RR }{\mathbb{R}}
\newcommand{\QQ }{\mathbb{Q}}
\newcommand{\ZZ }{\mathbb{Z}}
\newcommand{\FF }{\mathbb{F}}

\newcommand{\suchthat}{\;\ifnum\currentgrouptype=16 \middle\fi|\;}

\DeclareMathOperator{\initial}{in}

\DeclareMathOperator{\Conv}{Conv}

\DeclareMathOperator{\Trop}{Trop}
\DeclareMathOperator{\trop}{trop}

\DeclareMathOperator{\Det}{Det}

\DeclareMathOperator{\Grass}{Grass}
\DeclareMathOperator{\codim}{codim}
\DeclareMathOperator{\Lin}{Lin}

\newcommand{\lang}[1]{}
\newcommand{\kurz}[1]{#1}
\newcommand{\bmath}{\kurz{\begin{math}}\lang{\begin{displaymath}} }
    \newcommand{\emath}{\kurz{\end{math}}\lang{\end{displaymath}} }

\usepackage{etoolbox}
\AtBeginEnvironment{pmatrix}{\setlength{\arraycolsep}{2pt}}
\newcommand{\facetstatsOne}{%
  \begin{pmatrix}
    2597&1073&904&349&155&6\\
    112&49&82&109&77&20\\
    0&0&0&0&10&0
  \end{pmatrix}
}

\newcommand{\facetstatsTwo}{%
  \begin{pmatrix}
    2650&1096&981&458&242&26\\
    59&26&5&0&0&0\\
    0&0&0&0&0&0
  \end{pmatrix}
}

\newcommand{\facetstatsThree}{%
  \begin{pmatrix}
    2695&1121&986&458&242&26\\
    14&1&0&0&0&0\\
    0&0&0&0&0&0
  \end{pmatrix}
}

\newcommand{\facetstatsFour}{%
  \begin{pmatrix}
    2565&1070&658&152&13&0\\
    144&43&328&288&175&26\\
    0&9&0&18&54&0
  \end{pmatrix}
}

\newcommand{\facetstatsFive}{%
  \begin{pmatrix}
    2682&1022&781&332&185&24\\
    27&100&205&126&57&2\\
    0&0&0&0&0&0
  \end{pmatrix}
}

\newcommand{\facetstatsSix}{%
  \begin{pmatrix}
    2411&771&581&231&100&0\\
    298&351&395&197&127&4\\
    0&0&10&30&15&22\\
  \end{pmatrix}
}

\newcommand{\facetstatsSeven}{%
  \begin{pmatrix}
    2680&1002&803&244&41&0\\
    29&120&183&206&152&26\\
    0&0&0&8&49&0
  \end{pmatrix}
}

\newcommand{\facetstatsEight}{%
  \begin{pmatrix}
    2674&1025&864&421&207&26\\
    35&97&122&37&35&0\\
    0&0&0&0&0&0
  \end{pmatrix}
}

\newcommand{\facetstatsNine}{%
  \begin{pmatrix}
    2708&1122&986&458&242&26\\
    0&0&0&0&0&0\\
    1&0&0&0&0&0
  \end{pmatrix}
}

\newcommand{\facetstatsTen}{%
  \begin{pmatrix}
    2666&1109&986&458&242&26\\
    43&13&0&0&0&0\\
    0&0&0&0&0&0
  \end{pmatrix}
}

\newcommand{\facetstatsEleven}{%
  \begin{pmatrix}
    2591&1087&925&360&167&22\\
    118&35&61&98&75&4\\
    0&0&0&0&0&0
  \end{pmatrix}
}

\newcommand{\facetstatsTwelve}{%
  \begin{pmatrix}
    2567&892&710&209&14&0\\
    142&226&275&239&179&26\\
    0&4&1&10&49&0
  \end{pmatrix}
}

\begin{document}

\title{Computing tropical points and tropical links}
\author{Tommy Hofmann}
\address{Technische Universit\"at Kaiserslautern\\
  Fachbereich Mathematik\\
  Postfach 3049\\
  67663 Kaiserslautern\\
  Germany
}
\email{thofmann@mathematik.uni-kl.de}
\urladdr{http://www.mathematik.uni-kl.de/\textasciitilde thofmann}
\author{Yue Ren}
\address{Max-Planck-Institut f\"ur Mathematik in den Naturwissenschaften\\
  Inselstr. 22\\
  04103 Leipzig\\
  Germany
}
\email{yue.ren@mis.mpg.de}
\urladdr{https://www.yueren.de}

\thanks{The second author was partially supported by the DFG Priority Programme 1489 ``Algorithmic and Experimental Methods in Algebra, Geometry and Number Theory'' and the Center of Advanced Studies in Mathematics of Ben-Gurion University. This work was completed during the program ``Tropical Geometry, Amoebas and Polytopes'' at the Institute Mittag-Leffler. The second author would like to thank the institute for its hospitality.}

\subjclass[2010]{14T05, 52B20, 12J25, 13P15}
\keywords{Tropical geometry; Tropical variety; Tropical Grassmannian; Computer algebra;  Newton polygon}
\date{August 2018}

\maketitle

\begin{abstract}
  We present an algorithm for computing zero-dimensional tropical varieties based on triangular decomposition and Newton polygon methods. From it, we derive algorithms for computing points on and links of higher-dimensional tropical varieties, using intersections with affine hyperplanes to reduce the dimension to zero. We use the algorithms to show that the tropical Grassmannians $\mathcal G_{3,8}$ and $\mathcal G_{4,8}$ are not simplicial.
  % In this article, we construct fast algorithms for computing non-trivial points on and codimension one links of tropical varieties based on triangular decomposition and Newton polygons.
  % Using them, we show that the tropical Grassmannians $\mathcal G_{3,8}$ and $\mathcal G_{3,9}$ are not simplicial and have links of higher valency.
\end{abstract}

%%%%%%%%%%%%%%%%%%%%%%%%%%%%%%%%%%%%%%%%%%%%%%%%%%%%%%%%%%%%%%%%%%%%%%%%%%%%%%%%%%
\section{Introduction}

Given an affine variety $X$ over an algebraically closed field $K$ with non-trivial valuation, its tropical variety $\Trop(X)$ is the euclidean closure of its image under component-wise valuation. Tropical varieties arise naturally in many applications in mathematics \cite{ABGJ14,Mikhalkin05} and beyond, such as in the context of phylogenetic trees in biology \cite[Section 4]{PachterSturmfels05}, product-mix auctions in economics \cite{TranYu15} or finiteness of central configurations in the $5$-body problem in physics \cite{HamptonJensen11}.

% Explicit computation of tropical varieties is not only of interest for practical applications, but also of theoretical importance on many occasions \cite{SpeyerSturmfels,Herrmann2009,BJMS15,CM16}.

Nevertheless, computing tropical varieties is an algorithmically challenging task, requiring sophisticated techniques from computer algebra and convex geometry. The first algorithms were developed by Bogart, Jensen, Speyer, Sturmfels and Thomas \cite{BJSST07} for the field of complex Puiseux series $\CC\{\!\{t\}\!\}$. More recently, Chan and Maclagan introduced a new notion of Gr\"obner bases for general fields with valuation in order to compute tropical varieties thereover~\cite{CM13}. Concurrently, Chan developed a special algorithm for computing tropical curves \cite[Chapter 4]{Chan13}. All these algorithms have been implemented in \textsc{gfan}~\cite{gfan}, which is the currently most widely used program for computing tropical varieties. In this article, we touch upon two problems that arise in the computation.

The first problem is to pinpoint a \textit{tropical starting point}, a first point on the tropical variety from which all further computations start off. At present, the default is to traverse the Gr\"obner complex randomly while checking all vertices along the way for containment in the tropical variety. This is a rather inefficient approach however, as there can be significantly more Gr\"obner polyhedra outside the tropical variety than inside \cite[Theorem 6.3]{BJSST07}. % and the traversal is done randomly.
% Currently, Chan and Jensen are working on a new algorithm building on Chan's thesis on computing tropical curves via coordinate projections \cite{Chan13} and Jensen and Yu's work on stable intersections of tropical varieties \cite{JY16}.
%The second problem, which arises repeatedly, is to compute \textit{tropical links}, tropical varieties of simpler combinatorial structure which describe the original tropical variety locally. Their special structure allows them to be computed via tropical prevarieties by successively adding new elements to the generating set until it becomes a tropical basis. While this has proven to be successful for a wide range of examples, experiments show that with increasing input size the tropical prevariety computations become intractable.
The second problem, which arises repeatedly, is to compute \textit{tropical links}, tropical varieties of simpler combinatorial structure which describe the original tropical variety locally. Their special structure allows them to be computed via tropical prevarieties. While this has proven to be successful for a wide range of examples, experiments show that with increasing input size the tropical prevariety computations become intractable.

We present a simple yet novel approach for solving the aforementioned problems, based on the following bread-and-butter techniques in computer algebra and number theory:
\begin{enumerate}[leftmargin=*]
\item intersection with random hyperplanes,
\item triangular decomposition of zero-dimensional polynomial ideals,
\item reading off valuations of roots from Newton polygons.
\end{enumerate}
Moreover, the algorithm for tropical links also relies on a generalization of the \emph{Transverse Intersection Lemma} \cite[Lemma 3.2]{BJSST07} to general fields with valuation, which follows from recent results by Osserman and Payne~\cite{OssermanPayne13}.

We use our algorithms to study some higher tropical Grassmannians $\mathcal G_{k,n}$. They were first studied by Speyer and Sturmfels~\cite{SpeyerSturmfels}, who showed that $\mathcal G_{2,n}$ for $n\geq 2$ and $\mathcal G_{3,6}$ are simplicial fans, the former using an intriguing connection to spaces of phylogenetic trees and the latter through explicit computation. Additionally, in their work on the parametrization and realizability of tropical planes~\cite{Herrmann2009}, Hermannn, Jensen, Joswig and Sturmfels showed that $\mathcal G_{3, 7}$ is also a simplicial fan. We will complement these findings by showing that this does not hold for $\mathcal G_{3,8}$ and~$\mathcal G_{4,8}$.

All algorithms presented in this article have been implemented in the \textsc{Singular} library \texttt{tropicalNewton.lib} \cite{singular,tropicalnewtonmethodslib}, and are publicly available as part of the official \textsc{Singular} distribution. For computations in convex geometry, it relies on an interface to \textsc{gfanlib} \cite{gfan,gfanlibso}.

\begin{convention}
  % For the remainder of the article, let $K$ be an algebraically closed field with non-trivial valuation $\nu\!:K\rightarrow\RR\cup\{\infty\}$, though we will mainly focus on its restriction $\nu\!:K^\ast\rightarrow\RR$. Letting $\Gamma_\nu:=\nu(K^\ast)$ denote its value group, we assume that $1\in\Gamma_\nu$. As $K$ is algebraically closed, the surjection $\nu:K^{\ast}\rightarrow \Gamma_\nu$ splits, i.e. there exists a homomorphism $\psi:(\Gamma_\nu,+)\rightarrow (K^\ast,\cdot)$ with $\nu(\psi(w))=w$ \cite[Lemma 2.1.15]{MaclaganSturmfels}. We will fix one such $\psi$ and use $p^w$ to denote the element $\psi(w)\in K^\ast$, or $t^w$ if $K$ is the field of Puiseux series $\CC\{\!\{t\}\!\}$. Let $\mathfrak K$ denote the residue field of $K$.
  For the remainder of the article, let $K$ be an algebraically closed field with non-trivial valuation $\nu\!:K\!\rightarrow\RR\cup\{\infty\}$, though we will mainly focus on its restriction $\nu\!:K^\ast\!\rightarrow\RR$. We assume that $1\in\nu(K^\ast)$. As $K$ is algebraically closed, there exists a homomorphism $\psi:(\nu(K^\ast),+)\rightarrow (K^\ast,\cdot)$ with $\nu(\psi(w))=w$ \cite[Lemma 2.1.15]{MaclaganSturmfels}. We will fix one such $\psi$ and use $p^w$ to denote the element $\psi(w)\in K^\ast$, or $t^w$ if $K$ is the field of Puiseux series $\CC\{\!\{t\}\!\}$. Let $\mathfrak K$ denote the residue field of $K$.

  Furthermore, we fix a multivariate polynomial ring $K[x]:=K[x_1,\dots,x_n]$. By abuse of notation, we will also use $\nu$ to refer to the component-wise valuation $(K^\ast)^n\rightarrow\RR^n$.
\end{convention}

\subsubsection*{Acknowledgements}
  The authors would like to thank Michael Joswig and Benjamin Schr\"oter for their feedback on a previous version of the article.

\renewcommand{\emph}[1]{\textit{\textcolor{red}{#1}}}
\section{Computing zero-dimensional tropical varieties}

In this section we present an algorithm, Algorithm~\ref{alg:tropicalPoint0}, for computing zero-dimensional tropical varieties using triangular decomposition and Newton polygon methods. For the sake of simplicity, we restrict ourselves to the task of computing a single point on the tropical variety, as the structure of the algorithm easily suggests how the entire tropical variety can be computed with proper bookkeeping. We conclude the section by showing that any generic triangular set admits what we call a tree of unique Newton polygons, which is the best case for our algorithm as it allows us to compute its tropical variety purely combinatorially, see Example~\ref{ex:seriesOfNewtonPolygons}.

\begin{definition}
  Let $w\in\RR^n$. For a polynomial $f=\sum_{\alpha\in\NN^n} c_\alpha\cdot x^\alpha\in K[x]$, we define the evaluation of its tropicalization at $w$ to be
  \[ \trop(f)(w):=\min\{w\cdot \alpha+\nu(c_\alpha)\mid c_\alpha\neq 0 \}, \]
  and its \emph{initial form} with respect to $w$ to be
  \[ \initial_w(f)=\textstyle\sum_{w\cdot \alpha+\nu(c_\alpha)=\trop(f)(w)} \overline{c_\alpha p^{-\nu(c_\alpha)}}\cdot x^\alpha\in\mathfrak K[x]. \]
  For an ideal $I\unlhd K[x]$, we define its \emph{initial ideal} with respect to $w$ to be
  \[ \initial_w(I)=\langle \initial_w(f)\mid f\in I\rangle \unlhd \mathfrak K[x]. \]
  The \emph{tropical variety} of $I$ is then given by
  \begin{align*}
    \Trop(I) := &\{w\in\RR^n\mid \initial_w(I) \text{ monomial free}\}.%\\
    %=&\overline{\nu(V(I)\cap (K^\ast)^n)}.
  \end{align*}
  For single polynomials $f\in K[x]$ and finite subsets $F\subseteq K[x]$, we abbreviate $\Trop(f):=\Trop(\langle f\rangle)$ and $\Trop(F):=\Trop(\langle F\rangle)$.

  The tropical variety is naturally covered by Gr\"obner polyhedra and hence the support of a subcomplex of the Gr\"obner complex \cite[Theorem 3.3.2]{MaclaganSturmfels}. Its \emph{dimension} resp. \emph{lineality space} is the dimension resp. lineality space of that subcomplex. % polyhedral complex covering it. % The latter is also commonly referred to as homogeneity space in the constant coefficient case over the field $\CC\{\!\{t\}\!\}$.
\end{definition}

While the previous algorithms mainly work with the aforementioned definition of tropical varieties, the algorithms in this article focus on the following characterization:

\begin{theorem}[{\cite[Theorem 3.2.5]{MaclaganSturmfels}}]
  For any ideal $I\unlhd K[x]$ and its corresponding affine variety $X=V(I)\subseteq K^n$ we have
  \[\Trop(I) = \overline{\nu(X \cap (K^\ast)^n)},\]
  where $\overline{(\cdot)}$ denotes the closure in the euclidean topology.
\end{theorem}

We now describe how to exploit this geometric characterization algorithmically using triangular sets and Newton polygon methods.

\begin{definition}
  A set $F=\{f_1,\dots,f_n\}\subseteq K[x]$ is called \emph{triangular}, if for each $k=1,\dots,n$ we have $f_k\in K[x_1,\dots,x_k]$ of the form
  \[ f_k=c_k x_k^{d_k}+\text{terms of lower $x_k$-degree} \]
  for some $c_k\in K^\ast$ and $d_k\in\NN_{>0}$.
\end{definition}

\begin{proposition}[{\cite[Corollary 4.7.4]{GreuelPfister}}]\label{prop:laz}
  Let $I$ be a zero-dimensional ideal, then there exist triangular sets $F_1,\dots,F_s$ such that
  \[ \sqrt I = \textstyle\bigcap_{i=1}^s \sqrt{\langle F_i\rangle} \quad \text{and}\quad \langle F_i\rangle + \langle F_j \rangle = \langle 1 \rangle \text{ for }i\neq j.\]
\end{proposition}

\begin{remark}
  Triangular decompositions as in Proposition~\ref{prop:laz} were initially introduced by Lazard \cite{Lazard} for polynomial system solving. They are a weaker notion of a primary decomposition and can be obtained easier through various methods, see \cite[Procedure 1]{Lazard} or \cite[Algorithm~4.7.8]{GreuelPfister} for details.
\end{remark}

% \begin{proposition}
%   For a triangular set $F$ we have
%   \[ F \text{ monomial free}\quad\Longleftrightarrow\quad \langle F\rangle \text{ contains no monomial}.\]
% \end{proposition}

% With the notation of Proposition~\ref{prop:laz} we have $V(I) = \bigcup_{i=1}^s V(F_i)$, hence $\nu(V(I)) = \bigcup_{i=1}^s \nu(V(F_i))$. When looking for points on $\Trop(I)$, we can therefore restrict our attention to ideals with generators forming a triangular set.
% The second ingredient in our new approach are Newton polygons.

\begin{definition}
  For a univariate polynomial $f=\sum_{i=0}^d c_i \cdot x_k^i\in K[x_k]$, $c_i\in K$, the \emph{Newton polygon} or \emph{extended Newton polyhedron} is defined to be
  \[ \Delta(f):= \Conv(\{(i,\nu(c_i))\mid c_i\neq 0\})+(\{0\}\times \RR_{\geq 0}).\]
  Similarly, for a multivariate polynomial $f=\sum_{i=0}^d f_i\cdot x_k^{i}\in K[x_1,\dots,x_k]$, $f_i\in K[x_1,\dots,x_{k-1}]$, and a weight $w\in\RR^{k-1}$, we define the \emph{expected Newton polygon} of $f$ at $w$ to be
  \[ \Delta_w(f):=\Conv(\{(i,\trop(f_i)(w))\mid f_i\neq 0\})+(\{0\}\times \RR_{\geq 0}). \]
  We say $f$ has a \emph{unique Newton polygon} at $w$, if the initial form $\initial_w(f_i)$ is a monomial for all vertices $(i,\trop(f_i)(w)) \in \Delta_w(f)$.
  % Let $(x_i,y_i) \in \RR^2$, $1 \leq i \leq r$, be the vertices of $\Delta(f)$, ordered such that $x_1 < \dotsb < x_r$. We define $\Lambda(f) := \{ (y_{i} - y_{i + 1})/(x_{i+1} - x_i) \mid 1 \leq i \leq r - 1\}$ to be the set of negative values of the slopes of $\Delta(f)$.
  Let $\Lambda(f)$ resp. $\Lambda_w(f)$ denote the sets consisting of the negatives of the slopes of $\Delta(f)$ resp. $\Delta_w(f)$.
\end{definition}

The following two propositions justify the utility of Newton polygons and the term ``unique Newton polygon''.

\begin{proposition}[{\cite[Proposition II.6.3]{Neukirch}}]\label{prop:neukirch}
  Let $f$ be a univariate polynomial over $K$. Then $\Lambda(f)=\Trop(f)$.
\end{proposition}

\begin{proposition}\label{prop:uniqueNewtonPolygons}
  For a polynomial $f\in K[x_1,\dots,x_k]$ and a weight $w\in \RR^{k-1}$ the following are equivalent:
  \begin{samepage}
    \begin{enumerate}[leftmargin=*]
    \item $f$ has a unique Newton polygon at $w$,
    \item for all $z\in K^{k-1}$ with $\nu(z)=w$ we have $\Delta(f(z,x_k)) = \Delta_w(f)$.
    \end{enumerate}
  \end{samepage}
\end{proposition}
\begin{proof}
  Note that for any coefficient $c\in K$, any substitute $z\in K^{k-1}$ with $\nu(z)=w$ and any exponent vector $\alpha\in\NN^{k-1}$ we have $\nu(c\cdot z^\alpha) = w \cdot \alpha + \nu(c)=\trop(c\cdot x^\alpha)(w)$. Hence for any $f_i\in K[x_1,\dots,x_{k-1}]$ we always have
  \[ \nu(f_i(z)) \geq \trop(f_i)(w), \]
  with equality guaranteed if $\initial_w(f_i)$ is a monomial, i.e. $(1)$ implies $(2)$.

  For the converse, it suffices to show that the equality is guaranteed only if $f_i$ is a monomial. Since $K$ is algebraically closed, so is its residue field $\mathfrak K$. In particular, if $\initial_w(f_i)$ is no monomial, then it has a non-zero root in $\mathfrak K^{k-1}$. Picking any $z\in K^{k-1}$ with $\nu(z)=w$ and $\initial_w(f)(\overline{z_1\cdot p^{-\nu(z_1)}},\dots,\overline{z_{k-1}\cdot p^{-\nu(z_{k-1})}})=0$, $p\in K$ denoting a uniformizing parameter, yields $\nu(f_i(z))\gneq \trop(f_i)(w)$.
\end{proof}

\begin{example}
  Let $K=\overline{\QQ}_2$ be the algebraic closure of the $2$-adic numbers. The polynomial $f=2^3x_3^2+(x_1-x_2)x_3+(x_1^2-2x_2)\in K[x]$ has a unique Newton polygon at all $(w_1,w_2)\in\RR^2$ with $w_1\neq w_2$ and $2w_1\neq w_2+1$:

  For instance, given $(z_1,z_2) \in K^2$ with $\nu_2(z_1,z_2)=(2,1)$, the Newton polygon $\Delta(f(z_1,z_2,x_3))$ will have vertices at $(0,2)$, $(1,1)$ and $(2,3)$.
  Using Proposition~\ref{prop:neukirch} we conclude that $\Trop(f(z_1,z_2,x_3))=\{0,1\}$ and hence $(2,1,0)$, $(2,1,1) \in \Trop(f)$.
  On the other hand, for $(z_1,z_2) \in K^2$ with $\nu_2(z_1,z_2)=(0,0)$, the Newton polygon $\Delta(f(z_1,z_2,x_3))$ may vary depending on the choice of $z_1,z_2$, as illustrated in Figure~\ref{fig:nonUniqueNewtonPolygon}.
\end{example}
\vspace{-0mm}
\begin{figure}[ht]
  \centering
  \begin{tikzpicture}[yscale=0.6]
    \fill[blue!20] (0,2.25) -- (0,0)-- (1,0) -- (2,1.5) -- (2,2.25) -- cycle;
    \draw (0,2.25) -- (0,0)-- (1,0) -- (2,1.5) -- (2,2.25);
    \fill (0,0) circle (0.6mm);
    \fill (1,0) circle (0.6mm);
    \fill (2,1.5) circle (0.6mm);
    \node[anchor=north,font=\scriptsize] at (0,0) {$(0,0)$};
    \node[anchor=north,font=\scriptsize] at (1,0) {$(1,0)$};
    \node[anchor=north west,xshift=-0.1cm,yshift=0.1cm,font=\scriptsize] at (2,1.5) {$(2,3)$};
    \node[font=\small,yshift=-2mm] at (1,-1) {$\Delta_{(0,0)}(f)$};
    % \draw[->] (-0.5,1.25) -- (-1.5,1);
    % \draw[->] (3,1.25) -- (4,1);

    \node (o0) at (-3,-0) {};
    \fill[blue!20] ($(-1,2.25)+(o0)$) -- ($(-1,0)+(o0)$) -- ($(0,0.5)+(o0)$) -- ($(1,1.5)+(o0)$) -- ($(1,2.25)+(o0)$) -- cycle;
    \draw ($(-1,2.25)+(o0)$) -- ($(-1,0)+(o0)$) -- ($(0,0.5)+(o0)$) -- ($(1,1.5)+(o0)$) -- ($(1,1.5)+(o0)$) -- ($(1,2.25)+(o0)$);
    \fill ($(-1,0)+(o0)$) circle (0.6mm);
    \fill ($(0,0.5)+(o0)$) circle (0.6mm);
    \fill ($(1,1.5)+(o0)$) circle (0.6mm);
    \node[anchor=north,font=\scriptsize] at ($(-1,0)+(o0)$) {$(0,0)$};
    \node[anchor=north west,xshift=-0.1cm,yshift=0.1cm,font=\scriptsize] at ($(0,0.5)+(o0)$) {$(1,1)$};
    \node[anchor=north west,xshift=-0.1cm,yshift=0.1cm,font=\scriptsize] at ($(1,1.5)+(o0)$) {$(2,3)$};
    \node[font=\small,yshift=-2mm] at ($(0,-1)+(o0)$) {$\Delta(f(1,3,x_3))$};

    \node (o1) at (5,-0) {};
    \fill[blue!20] ($(-1,2.25)+(o1)$) -- ($(-1,0)+(o1)$) -- ($(1,1.5)+(o1)$) -- ($(1,2.25)+(o1)$) -- cycle;
    \draw ($(-1,2.25)+(o1)$) -- ($(-1,0)+(o1)$) -- ($(1,1.5)+(o1)$) -- ($(1,1.5)+(o1)$) -- ($(1,2.25)+(o1)$);
    \fill ($(-1,0)+(o1)$) circle (0.6mm);
    \fill ($(0,1)+(o1)$) circle (0.6mm);
    \fill ($(1,1.5)+(o1)$) circle (0.6mm);
    \node[anchor=north,font=\scriptsize] at ($(-1,0)+(o1)$) {$(0,0)$};
    \node[anchor=south,font=\scriptsize] at ($(0,1)+(o1)$) {$(1,2)$};
    \node[anchor=north west,xshift=-0.1cm,yshift=0.1cm,font=\scriptsize] at ($(1,1.5)+(o1)$) {$(2,3)$};
    \node[font=\small,yshift=-2mm] at ($(0,-1)+(o1)$) {$\Delta(f(1,5,x_3))$};
  \end{tikzpicture}\vspace{-5mm}
  \caption{the expected and possible Newton polygons of $f$.}
  \label{fig:nonUniqueNewtonPolygon}
\end{figure}
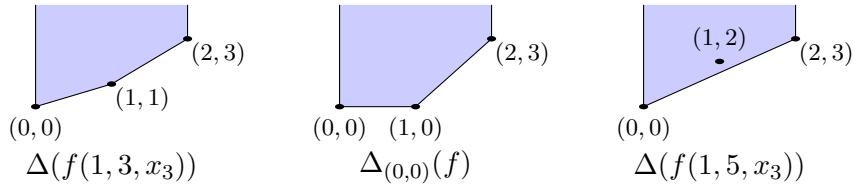
\vspace{-5mm}
\begin{algorithm}[tropical point, zero-dimensional case only]\label{alg:tropicalPoint0}\
  \begin{algorithmic}[1]
    \REQUIRE{$F=\{f_1,\dots,f_n\}\subseteq K[x]$ a triangular set with $V(F)\subseteq (K^\ast)^n$.}
    \ENSURE{$w\in \Trop(F)$.}
    \STATE Pick $w_1 \in \Lambda(f_1)$.
    \FOR{$i=2,\dots,n$}
    \IF{$f_i$ has a unique Newton polygon at $(w_1,\dots,w_{i-1})$}\label{step:newtonPolygon}
    \STATE Pick $w_i\in\Lambda_{(w_1,\dots,w_{i-1})}(f_i)$.
    \ELSE
    \STATE \label{step:rootComputation} Compute a root $(z_1,\dots,z_{i-1})\in V(f_1,\dots,f_{i-1})$.
    \STATE Pick $w_i\in\Lambda(f_i(z_1,\dots,z_{i-1},x_i))$.
    \ENDIF
    \ENDFOR
    \RETURN{$(w_1,\dots,w_n)$}
  \end{algorithmic}
\end{algorithm}
\begin{proof}
  The termination of the algorithm is clear and the correctness follows directly from Propositions~\ref{prop:neukirch} and \ref{prop:uniqueNewtonPolygons}. % However it needs to be shown that all occurring $\Lambda_{(w_1,\dots,w_{i-1})}(f_i)$ and $\Lambda(f_i(z_1,\dots,z_{i-1},x_i))$ are non-empty, i.e. that the respective Newton polygons have non-zero width.
  % This follows from the condition on $V(F)$: Suppose $f_i= c_i\cdot x_i^d+h_{d-1} x_i^{d-1}+\dots+h_0$ for some $h_{d-1},\dots,h_0\in K[x_1,\dots,x_{i-1}]\subseteq K[x_1,\dots,x_n]$ and some $d\in\NN_{>0}$. If $h_{d-1}(z)=\dots=h_{0}(z)=0$ for some $z\in V(f_1,\dots,f_{i-1})$, then $(z,0)\in V(f_1,\dots,f_i)$ contradicting $V(F)\subseteq (K^\ast)^n$. Hence $h_j(z)\neq 0$ for some $j<i$.
\end{proof}

While Algorithm~\ref{alg:tropicalPoint0} looks straightforward, performing Step~\ref{step:rootComputation} is a rather delicate task, which we will address in Examples~\ref{ex:rootApproximation} and \ref{ex:fieldExtension}. Example~\ref{ex:seriesOfNewtonPolygons} shows how Algorithm~\ref{alg:tropicalPoint0} can be used to compute the entire tropical variety.

% \begin{remarkNew}\hfill
%   \begin{enumerate}
%   \item
%     Although as described, Algorithm~\ref{alg:tropicalPoint0} computes only one point on a zero-dimensional tropical variety, it can easily be modified to compute all points on a zero-dimensional tropical variety: At Step~4 and Step~\ref{step:rootComputation} respectively one has to consider all $w_i \in \Lambda_{(w_1,\dotsc,w_{i-1})}(f_i)$ and $(c_1,\dotsc,c_{i-1}) \in V(f_1,\dotsc,f_{i-1})$ respectively.
%     For the sake of simplicity, we will refer to both versions as Algorithm~\ref{alg:tropicalPoint0}.
%   \item
%   Most of Algorithm~\ref{alg:tropicalPoint0} is straightforward, but performing Step~\ref{step:rootComputation} is a rather delicate task and depends very much on the field over which the polynomials are defined.
%   For example, if the polynomials in $F$ are given over a finite extension of a $p$-adic field $\QQ_p$, one can use Hensel Lifting to find its roots, however it might be necessary to pass to its splitting fields (see \cite{Yokoyama, GeisslerKluners}) to ensure their existence.
%   For all examples in this articles, this was uncessessary since either the Newton polygons were unique or the equations were of a form from which the roots could be easily read off.
%   \end{enumerate}
% \end{remarkNew}

\begin{example}[root approximation]\label{ex:rootApproximation}
  Note that, in Step~\ref{step:rootComputation} of Algorithm \ref{alg:tropicalPoint0}, it always suffices to approximate the root with respect to the metric induced by the valuation.
  For instance, consider the triangular set $F = \{f_1,f_2,f_3\} \subseteq \overline{\QQ}_3[x_1,x_2,x_3]$ with
  \[f_1 = x_1^2 + 3x_1 - 1, \quad f_2 = x_2^2 + 9x_2 - 1, \quad f_3 = 3x_3^2 + (x_1 - x_2)x_3 + 1.\]
  From the Newton polygons of $f_1$ and $f_2$ we see that elements $(z_1,z_2) \in (\overline{\QQ}_3)^2$ with $f_1(z_1) = f_2(z_1,z_2) = 0$ must satisfy $\nu_3(z_1,z_2) = (0,0)$.
  However, $f_3$ does not have a unique Newton polygon at $(0,0)$ and $\Delta(f_3(z_1,z_2,x_3))$ may vary depending on $z_1$ and $z_2$.
  More precisely, we have
  \[ \Delta(f_3(z_1,z_2,x_3)) = \begin{cases}
      \begin{tikzpicture}[scale=0.8]
      \fill[blue!20] (0,1.25) -- (0,0)-- (1,0) -- (2,1.0) -- (2,1.25) -- cycle;
      \draw (0,1.25) -- (0,0) -- node[above,font=\tiny] {$0$} (1,0) -- node[above,font=\tiny] {$1$} (2,1) -- (2,1.25);
      \fill (0,0) circle (0.6mm);
      \fill (1,0) circle (0.6mm);
      \fill (2,1.0) circle (0.6mm);
      \node[anchor=north,font=\tiny] at (0,0) {$(0,0)$};
      \node[anchor=north,font=\tiny] at (1,0) {$(1,0)$};
      \node[anchor=north west,xshift=-0.1cm,yshift=0.1cm,font=\tiny] at (2,1.0) {$(2,1)$};
      \node[anchor=west] at (3.5,0.5) {if $\nu_3(z_1 - z_2) = 0$,};
    \end{tikzpicture} \\
    \begin{tikzpicture}[scale=0.8]
      \fill[blue!20] (0,1.25) -- (0,0)-- (2,1.0) -- (2,1.25) -- cycle;
      \draw (0,1.25) -- (0,0) -- node[above,font=\tiny] {$\frac{1}{2}$} (2,1) -- (2,1.25);
      \fill (0,0) circle (0.6mm);
      \fill (2,1.0) circle (0.6mm);
      \node[anchor=north,font=\tiny] at (0,0) {$(0,0)$};
      \node[anchor=north west,xshift=-0.1cm,yshift=0.1cm,font=\tiny] at (2,1.0) {$(2,1)$};
      \node[anchor=west] at (3.5,0.5) {if $\nu_3(z_1 - z_2) > 0$.};
    \end{tikzpicture}\vspace{-4mm}
    \end{cases}
  \]
  Through Hensel Lifting we see that $f_1$ has a root $z_1 \in \ZZ_3$ with $z_1 \equiv 4 \bmod 3^2\ZZ_3$ and $f_2(z_1,x_2)$ has a root $z_2 \in \ZZ_3$ with $z_2 \equiv 1 \bmod 3^2\ZZ_3$.
  Since $z_1 - z_2\neq 0$ and $z_1 - z_2 \in 3 \ZZ_3$, we are in the second case and conclude that $(0, 0, -\frac 1 2) \in \Trop(F)$.
\end{example}

\begin{example}[field extensions]\label{ex:fieldExtension}
  While we began this article by fixing an algebraically closed field $K$, in practise we are always working over a finite extension of either the rationals $\QQ$, a finite field $\FF_q$ or function fields thereon. This can be problematic in conjunction with Step~\ref{step:rootComputation}, as approximating roots might require further field extensions. By the recursive nature of the algorithm, we potentially end up with a tower of field extensions. For instance, consider the triangular set $F = \{f_1,\dotsc,f_n\} \subseteq \QQ(\!(t)\!)[x_1,\dots,x_n]$ given by
  \[ f_k = x_k^2 - q_k t\Big( \sum_{i=1}^{k-1}x_i \Big) + q_k t^2, \text{ where }q_k\in\NN\text{ is the $k$-th prime}. \]
  This triangular set will never encounter a unique Newton polygon in Step~\ref{step:newtonPolygon}, and every root computation in Step~\ref{step:rootComputation} will require a new degree $2$ extension, as $V(f_1,\dotsc,f_k) \subseteq (\QQ(\sqrt {q_1},\dotsc, \sqrt{q_k})\{\!\{t\}\!\})^n \setminus(\QQ(\sqrt {q_1},\dotsc, \sqrt{q_{k-1}})\{\!\{t\}\!\})^n$. This eventually leads to a degree $2^n$ extension of $\QQ$, which shows in the performance of our implementation of Algorithm~\ref{alg:tropicalPoint0} in \texttt{tropicalNewton.lib}: computing the tropicalization for $n=13$ requires $8$ seconds and it roughly doubles with each increment of $n$. See Timings~\ref{ex:fieldExtensionCont} for a comparison with other algorithms.

\end{example}

\begin{example}[computing entire tropical varieties]\label{ex:seriesOfNewtonPolygons}
  As mentioned in the beginning of the section, Algorithm~\ref{alg:tropicalPoint0} can be used to compute entire tropical varieties of zero-dimensional ideals. This is done by computing a triangular decomposition as in Proposition~\ref{prop:laz} and applying the algorithm to each triangular set, while exhausting all in Steps $4$ and $6$-$7$.
  For instance, consider the triangular set $F = \{f_1, f_2, f_3\} \subseteq \CC\{\!\{t\}\!\}[x_1,x_2,x_3]$ with
  \begin{align*}
    f_1 =  tx_1^2+x_1+1, \quad f_2 = tx_2^2+x_1x_2+1, \quad f_3 = x_3+x_1x_2.
  \end{align*}
  Then $F$ admits several choices for slopes throughout the algorithm, and each choice in turn induces a new unique Newton polygon as illustrated in Figure~\ref{fig:seriesOfNewtonPolygons}.
  Keeping track of all of them, allows us to reconstruct its entire tropical variety:
  \[ \Trop(F) = \{ (0, 0, 0), (0, -1, -1), (-1, 1, 0), (-1, -2, -3)\}. \]
\end{example}
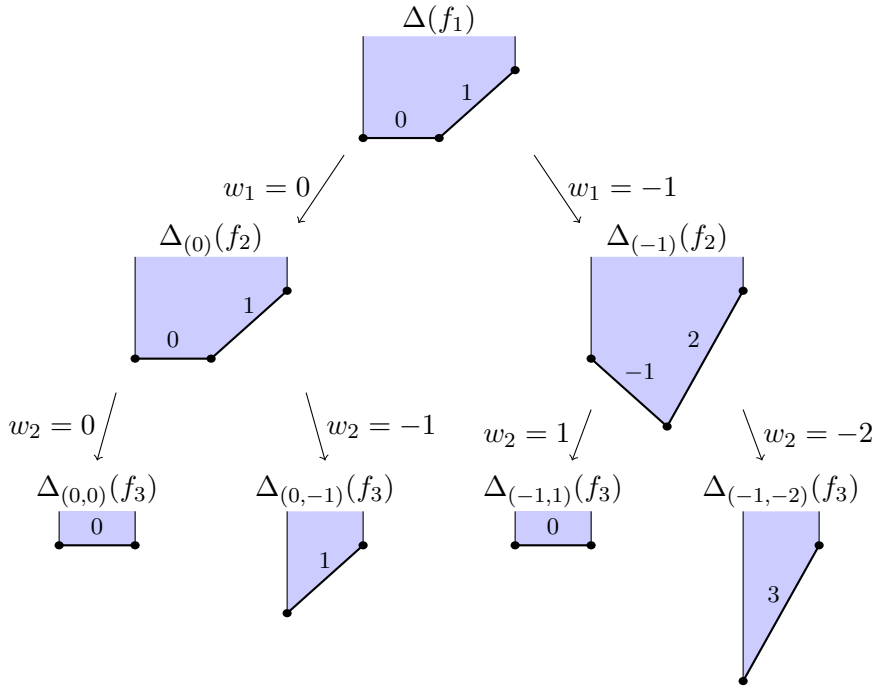
\begin{figure}[ht]
  \centering
  \begin{tikzpicture}[yscale=0.9]
    % first level
    \fill[blue!20] (-1,1.5) -- (-1,0) -- (0,0) -- (1,1) -- (1,1.5) -- cycle;
    \draw (-1,1.5) -- (-1,0)
    (1,1) -- (1,1.5);
    \draw[thick,black]
    (-1,0) -- node[above,black,font=\scriptsize] {$0$} (0,0);
    \draw[thick,black]
    (0,0) -- node[anchor=south east,xshift=0.1cm,yshift=-0.1cm,black,font=\scriptsize] {$1$} (1,1);
    \fill (-1,0) circle (0.6mm);
    \fill (0,0) circle (0.6mm);
    \fill (1,1) circle (0.6mm);
    \node[font=\small] at (0,1.75) {$\Delta(f_1)$};
    % \node[anchor=north,font=\scriptsize] at (-1,0) {$(0,0)$};
    % \node[anchor=north,font=\scriptsize] at (0,0) {$(1,0)$};
    % \node[anchor=north west,xshift=-0.1cm,yshift=0.1cm,font=\scriptsize] at (1,1) {$(2,1)$};

    \draw[->,black] (-1.25,-0.25) -- node[left,black,font=\small] {$w_1=0$} ++(-0.6,-1);
    \draw[->,black] (1.25,-0.25) -- node[right,black,font=\small] {$w_1=-1$} ++(0.6,-1);

    % second level
    \node (o11) at (-3,-3.25) {};
    \fill[blue!20] ($(-1,1.5)+(o11)$) -- ($(-1,0)+(o11)$) -- ($(0,0)+(o11)$) -- ($(1,1)+(o11)$) -- ($(1,1.5)+(o11)$) -- cycle;
    \draw ($(-1,1.5)+(o11)$) -- ($(-1,0)+(o11)$)
    ($(1,1)+(o11)$) -- ($(1,1.5)+(o11)$);
    \draw[thick,black]
    ($(-1,0)+(o11)$) -- node[above,black,font=\scriptsize] {$0$} ($(0,0)+(o11)$);
    \draw[thick,black]
    ($(0,0)+(o11)$) -- node[above,black,font=\scriptsize] {$1$} ($(1,1)+(o11)$);
    \fill ($(-1,0)+(o11)$) circle (0.6mm);
    \fill ($(0,0)+(o11)$) circle (0.6mm);
    \fill ($(1,1)+(o11)$) circle (0.6mm);
    \node[font=\small] at ($(0,1.75)+(o11)$) {$\Delta_{(0)}(f_2)$};
    % \node[anchor=north,font=\scriptsize] at ($(-1,0)+(o11)$) {$(0,0)$};
    % \node[anchor=north,font=\scriptsize] at ($(0,0)+(o11)$) {$(1,0)$};
    % \node[anchor=north west,xshift=-0.1cm,yshift=0.1cm,font=\scriptsize] at ($(1,1)+(o11)$) {$(2,1)$};

    \draw[->,black] ($(-1.25,-0.5)+(o11)$) -- node[left,black,font=\small] {$w_2=0$} ++(-0.25,-1);
    \draw[->,black] ($(1.25,-0.5)+(o11)$) -- node[right,black,font=\small] {$w_2=-1$} ++(0.25,-1);

    \node (o12) at (3,-3.25) {};
    \fill[blue!20] ($(-1,1.5)+(o12)$) -- ($(-1,0)+(o12)$) -- ($(0,-1) + (o12)$) -- ($(1,1)+(o12)$) -- ($(1,1.5)+(o12)$) -- cycle;
    \draw ($(-1,1.5)+(o12)$) -- ($(-1,0)+(o12)$)
    ($(1,1)+(o12)$) -- ($(1,1.5)+(o12)$);
    \draw[thick,black]
    ($(-1,0)+(o12)$) -- node[above,black,font=\scriptsize,xshift=4pt] {$-1$} ($(0,-1)+(o12)$) -- node[above,black,font=\scriptsize,xshift=-1.5mm] {$2$} ($(1,1)+(o12)$);

    \fill ($(-1,0)+(o12)$) circle (0.6mm);
    \fill ($(0,-1)+(o12)$) circle (0.6mm);
    \fill ($(1,1)+(o12)$) circle (0.6mm);
    \node[font=\small] at ($(0,1.75)+(o12)$) {$\Delta_{(-1)}(f_2)$};

    \draw[->,black] ($(-1,-0.75)+(o12)$) --  node[left,black,font=\small] {$w_2=1$} ++(-0.25,-0.75);
    \draw[->,black] ($(1,-0.75)+(o12)$) --  node[right,black,font=\small] {$w_2=-2$} ++(0.25,-0.75);
    % third level

    \node (o21) at ($(o11)+(-2,-2.75)$) {};
    \fill[blue!20] ($(0,0.5)+(o21)$) -- ($(0,0)+(o21)$) -- ($(1,0)+(o21)$) -- ($(1,0.5)+(o21)$) -- cycle;
    \draw ($(0,0.5)+(o21)$) -- ($(0,0)+(o21)$)
    ($(1,0)+(o21)$) -- ($(1,0.5)+(o21)$);
    \draw[thick,black]
    ($(0,0)+(o21)$) -- node[above,black,font=\scriptsize] {$0$} ($(1,0)+(o21)$);
    \fill ($(0,0)+(o21)$) circle (0.6mm);
    \fill ($(1,0)+(o21)$) circle (0.6mm);
    \node[font=\small] at ($(0.5,0.8)+(o21)$) {$\Delta_{(0,0)}(f_3)$};
    % \node[anchor=north,font=\scriptsize] at ($(0,0)+(o21)$) {$(0,0)$};
    % \node[anchor=north,font=\scriptsize] at ($(1,0)+(o21)$) {$(1,0)$};

    \node (o22) at ($(o11)+(1,-2.75)$) {};
    \fill[blue!20] ($(0,0.5)+(o22)$) -- ($(0,-1)+(o22)$) -- ($(1,0)+(o22)$) -- ($(1,0.5)+(o22)$) -- cycle;
    \draw ($(0,0.5)+(o22)$) -- ($(0,-1)+(o22)$)
    ($(1,0)+(o22)$) -- ($(1,0.5)+(o22)$);
    \draw[thick,black]
    ($(0,-1)+(o22)$) -- node[above,black,font=\scriptsize] {$1$} ($(1,0)+(o22)$);
    \fill ($(0,-1)+(o22)$) circle (0.6mm);
    \fill ($(1,0)+(o22)$) circle (0.6mm);
    \node[font=\small] at ($(0.5,0.8)+(o22)$) {$\Delta_{(0,-1)}(f_3)$};
    % \node[anchor=north east,font=\scriptsize] at ($(0,0)+(o22)$) {$(0,0)$};
    % \node[anchor=north west,xshift=-0.1cm,yshift=0.1cm,font=\scriptsize] at ($(1,-1)+(o22)$) {$(1,-1)$};

    \node (o23) at ($(o12)+(-2,-2.75)$) {};
    \fill[blue!20] ($(0,0.5)+(o23)$) -- ($(0,0)+(o23)$) -- ($(1,0)+(o23)$) -- ($(1,0.5)+(o23)$) -- cycle;
    \draw ($(0,0.5)+(o23)$) -- ($(0,0)+(o23)$)
    ($(1,0)+(o23)$) -- ($(1,0.5)+(o23)$);
    \draw[thick,black]
    ($(0,0)+(o23)$) -- node[above,black,font=\scriptsize] {$0$} ($(1,0)+(o23)$);
    \fill ($(0,0)+(o23)$) circle (0.6mm);
    \fill ($(1,0)+(o23)$) circle (0.6mm);
    \node[font=\small] at ($(0.5,0.8)+(o23)$) {$\Delta_{(-1,1)}(f_3)$};
    % \node[anchor=north,font=\scriptsize] at ($(0,0)+(o23)$) {$(0,0)$};
    % \node[anchor=north,font=\scriptsize] at ($(1,0)+(o23)$) {$(1,0)$};

    \node (o24) at ($(o12)+(1,-2.75)$) {};
    \fill[blue!20] ($(0,0.5)+(o24)$) -- ($(0,-2)+(o24)$) -- ($(1,0)+(o24)$) -- ($(1,0.5)+(o24)$) -- cycle;
    \draw ($(0,0.5)+(o24)$) -- ($(0,-2)+(o24)$)
    ($(1,0)+(o24)$) -- ($(1,0.5)+(o24)$);
    \draw[thick,black]
    ($(0,-2)+(o24)$) -- node[above,black,xshift=-1mm,font=\scriptsize] {$3$} ($(1,0)+(o24)$);
    \fill ($(0,-2)+(o24)$) circle (0.6mm);
    \fill ($(1,0)+(o24)$) circle (0.6mm);
    \node[font=\small] at ($(0.5,0.8)+(o24)$) {$\Delta_{(-1,-2)}(f_3)$};
    % \node[anchor=north,font=\scriptsize] at ($(0,0)+(o24)$) {$(0,0)$};
    % \node[anchor=north west,xshift=-0.1cm,yshift=0.1cm,font=\scriptsize] at ($(1,2)+(o24)$) {$(1,2)$};
  \end{tikzpicture}\vspace{-5mm}
  \caption{a tree of unique Newton polygons.}
  \label{fig:seriesOfNewtonPolygons}
\end{figure}

We conclude this section by showing that any generic triangular set resembles Example~\ref{ex:seriesOfNewtonPolygons} in the sense that its tropical variety is determined by a tree of unique Newton polygons.

\begin{definition}
  We say a triangular set $F=\{f_1,\dots,f_n\}\subseteq K[x]$ admits a \emph{tree of unique Newton polygons}, if for all $k=1,\dots,n$ and all weights $w=(w_1,\dots,w_{k-1})\in \RR^{k-1}$ with $w_i \in \Lambda_{(w_1,\dots,w_{i-1})}(f_i)$, $i=1,\dots,k-1$, the polynomial $f_k$ has a unique Newton polygon at $(w_1,\dots,w_{k-1})$.
\end{definition}

\begin{lemma}\label{lem:tropical}
  Consider $w\in \RR^{k-1}$ and $f\in K[x_1,\dots,x_k]\subseteq K[x_1,\dots,x_n]$ such that $(\{w\}\times\RR^{n-k+1}) \cap \Trop(f)$ has codimension $k$. Then $f$ has a unique Newton polygon at $w$ and
  \[ (\{w\}\times\RR^{n-k+1}) \cap \Trop(f) = \{w\}\times\bigcup_{\tilde w\in\Lambda_w(f)} \{\tilde w\}\times\RR^{n-k}.\]
\end{lemma}
\begin{proof}
  Without loss of generality, assume that $k=n$.
  Suppose $f=\sum_{i=0}^d f_i\cdot x_k^i$ with $f_i\in K[x_1,\dots,x_{k-1}]$ and assume that $f$ has no unique Newton polygon at $w$, i.e., that there exists a vertex $(i,\trop(f_i)(w))\in\Delta_w(f)$ such that $\initial_{w}(f_i)$ is no monomial. Let $\mu_0$ and $\mu_1$ be the negated slopes of the edges after and before the vertex respectively, see Figure~\ref{fig:nonmonomialVertex}.
  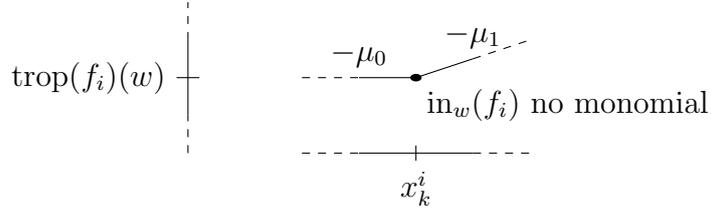
\begin{figure}[ht]
    \centering
    \begin{tikzpicture}[xscale=1.5]
      \draw[dashed] (1,1) -- (1.5,1);
      \draw (1.5,1) -- (2,1);
      \draw (2,1) -- (2.5,1.25);
      \draw[dashed] (2.5,1.25) -- (3,1.5);
      \fill[black] (2,1) circle (0.5mm);

      \draw (1.5,0) -- (2.5,0);
      \draw[dashed] (1,0) -- (1.5,0);
      \draw[dashed] (2.5,0) -- (3,0);
      \draw (2,0.1) -- (2,-0.1);
      \node[anchor=north] at (2,-0.1) {$x_k^i$};

      \draw (0,0.5) -- (0,1.5);
      \draw[dashed] (0,0) -- (0,0.5);
      \draw[dashed] (0,1.5) -- (0,2);
      \draw (0.1,1) -- (-0.1,1);
      \node[anchor=east] at (-0.1,1) {$\trop(f_i)(w)$};

      \node[anchor=south] at (1.5,1) {$-\mu_0$};
      \node[anchor=south] at (2.5,1.25) {$-\mu_1$};
      \node[anchor=north west] at (2,1) {$\initial_{w}(f_i)$ no monomial};
    \end{tikzpicture}\vspace{-5mm}
    \caption{$\Delta_w(f)$ around a non-monomial vertex}
    \label{fig:nonmonomialVertex}
  \end{figure}

  Then, for any $w_k\in(\mu_0,\mu_1)$, we have $\initial_{(w,w_k)}(f)=\initial_{w}(f_i)\cdot x_k^i$, which is no monomial. This implies $\{w\}\times(\mu_1,\mu_0)\subseteq\Trop(f)$, contradicting the zero-dimensionality of $\Trop(f)$.

  Next, we show the equality. For the ``$\supseteq$'' inclusion, let $\mu$ be a slope of an edge of $\Delta_w(f)$, say connecting the two vertices $v_0$ and $v_1$. Then, writing $e(v_0,v_1)$ for the edge connecting $v_0$ and $v_1$,
  \[ \initial_{(w,\mu)}(f) = \textstyle\sum_{(i,\trop(f_i)(w))\in e(v_0,v_1)}\initial_{w}(f_i)\cdot x_k^i.\]

  For the converse inclusion, let $(w,w_k)\in\Trop(f)$. It is clear, that for some bounded proper face $e\leq\Delta_{w}(f)$,
  \[ \initial_{(w,w_k)}(f) = \textstyle\sum_{(i,\trop(f_i)(w))\in e}\initial_{w}(f_i)\cdot x_k^i.\]
  Note that $e$ cannot be zero-dimensional, as otherwise $\initial_{(w,w_k)}(f)=\initial_{(w,w_k')}(f)$ for all $w_k'\in\RR$, contradicting the zero-dimensionality of $\Trop(f)$. Hence, $e$ has to be an edge and, consequently, $w_k$ is the slope of $e$.
\end{proof}

\begin{proposition}\label{prop:uniqueNewtonPolygons2}
  For a triangular set $F=\{f_1,\dots,f_n\}\subseteq K[x]$ the following are equivalent:
  \begin{enumerate}
  \item $\dim \bigcap_{i=1}^k \Trop(f_i) = n-k$ for all $k=1,\dots,n$,
  \item $F$ is a tropical basis.
  \end{enumerate}
  Moreover, if $F$ is a tropical basis, then it admits a tree of unique Newton polygons.
\end{proposition}
\begin{proof}
  We first show that $(1)$ implies that $F$ is a tropical basis and that it admits a tree of unique Newton polygons.
  By definition, we have $\Trop(f_1)=\bigcup_{w_1\in\Lambda(f_1)}\{w_1\}\times\RR^{n-1}$. Applying Lemma~\ref{lem:tropical} repeatedly, we see that for $w_1\in\Lambda(f_1)$, the polynomial $f_2$ has a unique Newton polygon at $w_1$ with
  \[ \left(\{w_1\}\times\RR^{n-1}\right) \cap \Trop(f_2) = \{w_1\} \times \bigcup_{w_2\in\Lambda_{w_1}(f_2)} \{w_2\} \times\RR^{n-2}, \]
  and, for $w_1\in\Lambda(f_1)$ and $w_2\in\Lambda_{w_1}(f_2)$, $f_3$ has a unique Newton polygon at $(w_1,w_2)$ with
  \[ \left(\{(w_1,w_2)\}\!\times\!\RR^{n-2}\right) \cap \Trop(f_3) = \{(w_1,w_2)\} \times \bigcup_{w_3\in\Lambda_{(w_1,w_2)}(f_3)} \{w_3\} \times\RR^{n-3}, \]
  and so forth. This shows on the one hand that $F$ admits a tree of unique Newton polygons and on the other hand that any point in $\bigcap_{i=1}^n \Trop(f_i)$ corresponds to the component-wise valuation of a point in $V(F)$, implying that $F$ is a tropical basis.

  It remains to show that if $(1)$ is not true, then $F$ is no tropical basis. Assume for the sake of simplicity that $\dim \Trop(f_1)\cap\Trop(f_2) = n-1$. Because $\Trop(f_1)=\bigcup_{w_1\in\Lambda(f_1)}\{w_1\}\times\RR^{n-1}$ and $\Trop(f_2)$ is invariant under translation by $\{(0,0)\}\times\RR^{n-2}$, there necessarily exist
  \[ \{\lambda\}\times [\mu_1,\mu_2] \times \RR^{n-2}\subseteq \Trop(f_1)\cap\Trop(f_2), \]
  for $\lambda\in\Lambda(f_1)$ and a nontrivial $[\mu_1,\mu_2]\subseteq\RR$. Consequently,
  \[ \{\lambda\}\times[\mu_1,\mu_2]\times\{(0,\dots,0)\}\subseteq \bigcap_{i=1}^n \Trop(f_i), \]
  and since $\bigcap_{i=1}^n \Trop(f_i)$ is not zero-dimensional, $F$ cannot be a tropical basis of the zero-dimensional ideal it generates.
\end{proof}

From Proposition~\ref{prop:uniqueNewtonPolygons2}, we conclude that a generic triangular set is a tropical basis and admits a tree of unique Newton polygons in the following sense:

\begin{corollary}\label{cor:tropical}
  Let $(K^\ast)^N\subseteq K[x]^n$ be the coefficient space of all triangular sets with fixed support. Then, in the topology induced by the valuation, there exists an open dense set $\mathcal U\subseteq (K^\ast)^N$ such that any triangular set $F\in\mathcal U$ is a tropical basis and admits a tree of unique Newton polygons.
\end{corollary}
\begin{proof}
  Consider the component-wise valuation $\nu:(K^\ast)^N\rightarrow\RR^N$. There exists an euclidean open dense subset $U\subseteq\RR^n$ such that the tropical hypersurfaces of any triangular set $F\in (K^\ast)^N$ with $\nu(F)\in U$ intersect transversally as in Proposition~\ref{prop:uniqueNewtonPolygons2} (1). As $\nu$ is continuous, its preimage $\mathcal U:=\nu^{-1}U\subseteq (K^\ast)^N$ is also open and dense.
\end{proof}

\section{Computing tropical starting points}

In this section, we use Algorithm~\ref{alg:tropicalPoint0} to compute points on higher-dimensional tropical varieties. This is done by reducing the dimension to zero by intersecting with randomly chosen hyperplanes.
Moreover, we will use the algorithm to sample random maximal Gr\"obner cones on the tropical Grassmannians $\mathcal G_{3,7}, \mathcal G_{4,7}, \mathcal G_{3,8}, \mathcal G_{4,8}$ and show that the latter two are not simplicial.

% \begin{definition}
%   Let $I\unlhd K[x]$ be an ideal. The \emph{homogeneity space} of $I$ is given by
%   \[ C_0(I):=\{w\in\RR^n\mid \initial_w(I)=... \]
%   It can be read off any reduced Gr\"obner basis and is trivially included in $\Trop(I)$, provided $\Trop(I)$ is non-empty. We call a point $p\in \Trop(I)$ \emph{non-trivial}, if $p\notin C_0(I)$.
% \end{definition}

\begin{proposition}\label{prop:affineSubspaceIntersection}
  % Let $I\unlhd K[x]$ be a prime ideal of dimension $d$ with algebraically independent set $\{x_1,\dots,x_d\}$. Let $X:=V(I)$ be its affine variety and suppose $X\cap (K^\ast)^n\neq\emptyset$.
  Let $I \unlhd K[x]$ be a prime ideal of dimension $d$ and $X=V(I)$ its corresponding irreducible affine variety such that $X \cap (K^\ast)^n \neq \emptyset$. W.l.o.g. let $\{x_1,\dotsc,x_d\}$ be algebraically independent modulo $I$.
  Then there exists a non-empty, Zariski open subset $U\subseteq (K^\ast)^{d}$ such that for all $\lambda\in U$
  \[ \emptyset\neq X\cap V(\langle x_i-\lambda_i\mid i=1,\dots,d\rangle) \subseteq (K^\ast)^n \]
  and $\dim(X\cap V(\langle x_i-\lambda_i\mid i=1,\dots,d\rangle)=0$.
\end{proposition}
\begin{proof}
  Abbreviating $H_\lambda:=V(\langle x_i-\lambda_i\mid i=1,\dots,d\rangle)$,
  it is clear that there exists a Zariski open $U_0\subseteq (K^\ast)^{d}$ with $\emptyset\neq X\cap H_\lambda$ and $\dim(X\cap H_\lambda) = 0$.
  Now consider the set in which the inclusion does not hold. It naturally decomposes into $n-d$ subsets:\vspace{-1mm}
  \[ \vspace{-1mm} A:=\{\lambda\in (K^\ast)^d\mid X\cap H_\lambda\!\nsubseteq\! (K^\ast)^n \}
    =\!\bigcup_{i=d+1}^n\! \underbrace{\{ \lambda\in (K^\ast)^d\mid \exists z\!\in\! X\cap H_\lambda\!: z_i\!=\!0 \}}_{=:A_i}.\]
  As $U$ can be chosen to be $U_0 \setminus \overline A$, where $\overline{(\cdot)}$ denotes the Zariski closure in $(K^\ast)^d$,  it suffices to show that $\overline A_i \neq (K^\ast)^d$. This is easy to see: Because $X$ is irreducible and $X\cap (K^\ast)^n\neq\emptyset$, we necessarily have $\dim (X\cap V(x_i))<d$ for all $i=d+1,\dots,n$. In particular, $\dim \pi(X\cap V(x_i))<d$, where $\pi: K^n\twoheadrightarrow K^d$ is the canonical projection onto the first $d$ coordinates. And, by construction, $A_i\subseteq \pi(X\cap V(x_i))$.
\end{proof}

Proposition~\ref{prop:affineSubspaceIntersection} can be reformulated into the following algorithm.

\begin{algorithm}[tropical point]\label{alg:tropicalStartingPoint}\
  \begin{algorithmic}[1]
    \REQUIRE{$I\unlhd K[x]$ prime ideal with $V(I)\cap(K^\ast)^n\neq\emptyset$.}
    \ENSURE{$w\in\Trop(I)$.}
    \STATE Compute a maximal algebraically indep. set modulo $I$, say $\{x_1,\dots,x_d\}$.
    \REPEAT
%     \STATE Pick $w\in\QQ^d$ random with $\{w\}\times\RR^{n-d}\cap C_0(I)=\emptyset$.
    \STATE Pick $z=(z_1,\dots,z_d)\in (K^\ast)^d$ randomly.\label{algstep:random}
    \STATE Set $I_z$ to be the image of $I$ under the substitution map\vspace{-1.5mm}
    \[\vspace{-1.5mm} K[x_1,\dots,x_n]\rightarrow K[x_{d+1},\dots,x_n],\quad x_i\mapsto
    \begin{cases}
      z_i& \text{if }i\leq d,\\
      x_i& \text{else}.
    \end{cases} \]
    \UNTIL{$\dim(I_z)=0$ and $V(I_z)\subseteq (K^\ast)^{n-d}$}
    \STATE Compute a triangular set $F\subseteq K[x_{d+1},\dots,x_{n}]$ with $\sqrt{I_z} \subseteq \langle F\rangle$.
    \STATE Compute a point $(w_{d+1},\dots,w_n)\in\Trop(F)$ using Algorithm~\ref{alg:tropicalPoint0}.
    \RETURN{$(\nu(z),w_{d+1},\dots,w_n)$}
  \end{algorithmic}
\end{algorithm}

\begin{remark}
  \begin{enumerate}
  \item
    Randomized algorithms such as Algorithm~\ref{alg:tropicalStartingPoint} are commonly referred to as Las Vegas algorithms. This means that its result is always correct, however it only has an expected finite runtime. Nevertheless, Proposition~\ref{prop:affineSubspaceIntersection} shows that generic choices of $z$ in Step~\ref{algstep:random} will lead to termination.
  \item
    Note that the set of all $w\in\RR^d$ such that $\{w\}\times\RR^{n-d}$ does not intersect any lower-dimensional Gr\"obner polyhedra on $\Trop(I)$ is open and dense in the euclidean topology. Hence generic choices of $z\in(K^\ast)^d$ in Step~\ref{algstep:random} will also guarantee that the resulting tropical point will lie in the relative interior of a maximal Gr\"obner polyhedra on the tropical variety.
  \item
    It is possible to eliminate the randomness by computing stable intersections with affine hyperplanes, as in a recent work of Jensen and Yu \cite{JY16}. However, this requires one transcendental extension of $K$ per hyperplane, which is not feasible in high codimension.
  \end{enumerate}
\end{remark}

We will briefly define the examples of our interest.

\begin{definition}\label{def:grassmannians}
  Let $k,n \in \NN_{>0}$ with $1 \leq k \leq n$. The \emph{tropical Grassmannian} $\mathcal G_{k,n} \subseteq \RR^{{n \choose k}}$ is defined to be the tropicalization of the ideal $\Grass(k,n)\unlhd K[p]$, where the variables of the ring $K[p]:=K[p_{i_1\cdots i_k}\mid 1\leq i_1<\dots<i_k\leq n]$ represent the $k\times k$ minors of any $k\times n$ matrix and the ideal $\Grass(k,n)$ is generated by all Pl\"ucker relations amongst them, see \cite[Section 4.3]{MaclaganSturmfels}. We consider the variables of $K[p]$ to be sorted lexicographically, i.e.
  \[ p_{i_1\cdots i_k} > p_{j_1\cdots j_k} \quad:\Longleftrightarrow\quad \exists 1\leq l<k\!: i_1=j_1, \dots, i_{l-1}=j_{l-1} \text{ and } i_l > j_l.\]
  Moreover, we define the ideal $\Det(k,n)\unlhd K[x_{11},x_{12},\dots,x_{nn}]$ to be the ideal generated by the $k\times k$ minors of the matrix $(x_{ij})_{i,j=1,\dots,n}$.
\end{definition}

\begin{example}[$\mathcal G_{2,5}$]\label{ex:firstGrassmann}
  Let $K = \CC\{\!\{t\}\!\}$. We demonstrate Algorithm~\ref{alg:tropicalStartingPoint} on the tropical Grassmannian $\mathcal G_{2, 5}$. Its ideal is given by
  \begin{align*}
    \Grass(2,5) = \langle  & p_{34}p_{25}\! - \!p_{24}p_{35}\! + \!p_{23}p_{45}, p_{34}p_{15}\! - \!p_{14}p_{35}\! + \!p_{13}p_{45}, p_{24}p_{15}\! - \!p_{14}p_{25}\! + \!p_{12}p_{45}, \\
    & p_{23}p_{15}\! - \!p_{13}p_{25}\! + \!p_{12}p_{35}, p_{23}p_{14}\! - \!p_{13}p_{24}\! + \!p_{12}p_{34} \rangle\unlhd K[p_{12},p_{13},\dots,p_{45}].
  \end{align*}
  It is $7$-dimensional with maximal independent set $\{p_{15},p_{23},p_{24},\dots,p_{45}\}$.

  Choosing $(z_{15},z_{23},z_{24},\dots,z_{45}):=(t,\dots,t)$ yields $I_z = \langle p_{12},p_{13},p_{14} \rangle$, which means that the choice is not generic in the sense of Proposition~\ref{prop:affineSubspaceIntersection}. On the other hand choosing $(t, t^5, t^3, t^7, t^8, t^2, t^9)$  yields $I_z$ generated by the triangular set
  \[ p_{12} + (t^3 - 1),\quad t^6\cdot p_{13} + (t^7 - 1), \quad t^2\cdot p_{14} + (t^4 - 1). \]
  Looking at the Newton polygons, we conclude that $w:=(0,-6,-2)\in\Trop(I_z)$. Thus
  $(\nu(z),w)=(0, -6, -2, 1,5,3,7,8,2,9) \in \Trop(I)$.
\end{example}

In addition to computing starting points for the tropical traversals, Algorithm~\ref{alg:tropicalStartingPoint} can be used to sample random points on tropical varieties.

\begin{example}[$\mathcal G_{k, n}$ for $k\in\{3,4\}$ and $n\in\{7,8\}$]
  Using Algorithm~\ref{alg:tropicalStartingPoint}, we sampled random maximal cones on higher tropical Grassmannians ignoring symmetry. This was done by computing Gr\"obner cones around random tropical points, dismissing those of lower dimension and duplicates. We analyzed over $1000$ distinct maximal cones on each of $\mathcal G_{3, 7}, \mathcal G_{4, 7}$ and $\mathcal G_{3, 8}$, as well as over $100$ distinct maximal cones on the tropical variety of $\mathcal G_{4, 8}$.

  All cones were invariant under tensoring with $\mathbb F_2$, which is not surprising for $\mathcal G_{3,7}$: Even though Speyer and Sturmfels showed that $\mathcal G_{3,7}$ depends on the characteristic of the ground field, in fact it is the smallest tropical Grassmannian depicting this behavior \cite[Theorem 3.7]{SpeyerSturmfels}, Herrmann, Jensen, Joswig and Sturmfels showed that, out of the $252\,000$ maximal cones of $\mathcal G_{3,7}$, this is only visible on a single cone, the Fano cone \cite[Theorem 2.1]{Herrmann2009}.

  Of the $1000$ Gr\"obner cones sampled from each of $\mathcal G_{3,7}$ and $\mathcal G_{4,7}$, every single one was simplicial, which was expected as $\mathcal G_{3,7}$ is known to be simplicial \cite[Theorem 2.1]{Herrmann2009} and $\mathcal G_{4,7}=\mathcal G_{3,7}$ by duality. In the $1000$ and $100$ Gr\"obner cones sampled from $\mathcal G_{3,8}$ and $\mathcal G_{4,8}$ respectively, each contained exactly one cone which was not simplicial, see the proof of Theorem~\ref{thm:simpliciality}.

  Not much is known on $\mathcal G_{3,8}$ and $\mathcal G_{4,8}$, but there is a complete description of the \emph{Dressian} $\mathcal D_{3,8}$ by Herrmann, Joswig and Speyer \cite[Theorem 31]{Herrmann2014}, which is a natural tropical prevariety containing $\mathcal G_{3,8}$ that parametrizes all tropical linear spaces. It is known that all rays of $\mathcal D_{3,7}$ and $\mathcal D_{3,8}$ are also rays of $\mathcal G_{3,7}$ and $\mathcal G_{3,8}$ respectively, and that $\mathcal G_{3,7}$ contains rays which are not rays of $\mathcal D_{3,7}$. Our sampling also revealed that this holds for $\mathcal G_{3,8}$. In fact, none of the $126$ tested rays of $\mathcal G_{3,8}$ were rays of $\mathcal D_{3,8}$, a concrete example is the ray generated by the following vector:
  \begin{align*}
    (&0, -1, 1, 1, -1, 0, 0, 0, 0, 1, 0, 1, 1, 1, 1, 1, 0, 0, 1, 1, -2, 0, 0, -1, -1, 0, -1, -1,\\[-0.175cm]
     &0,-1,0, 0, 0, 0, 0, 0, 1, 1, 1, 1, 1, 0, 0, 1, 1, 0, 0, 1, 1, 0, 0, 0, 0, 0, 0, 0)\in\RR^{\binom{8}{3}}.
  \end{align*}
  This is somewhat in stark contrast to $\mathcal G_{3,7}$ and $\mathcal D_{3,7}$, as out of the $721$ rays of the Grassmannian $616$ were rays of the Dressian \cite[Theorem 2.2]{Herrmann2009}.
  % This suggests that it is possible for irregularities to exist only in very specific areas of a tropical variety and that random sampling might not be a reliable tool for analyzing tropical varieties.
\end{example}

As an immediate result, we obtain:

\begin{theorem}\label{thm:simpliciality}
  The tropical Grassmannian $\mathcal G_{d,n}$ is not simplicial for $d=3,4$ and $n=8$.
\end{theorem}
\begin{proof}
  Consider the following two points which lie on $\mathcal G_{3,8}$ and $\mathcal G_{4,8}$ respectively:
  \begin{align*}
    w_{3,\!8}=(&2,\!10,\!7,\!10,\!2,\!2,\!2,\!10,\!7,\!10,\!9,\!6,\!9,\!12,\!12,\!12,\!9,\!6,\!9,\!12,\!5,\!2,\!5,\!5,\!2,\!5,\!5,\!2,\!5,\!1,\\[-1.5mm]
               &7,\!7,\!7,\!7,\!7,\!9,\!6,\!9,\!12,\!12,\!12,\!9,\!6,\!9,\!12,\!9,\!7,\!9,\!12,\!7,\!7,\!7,\!7,\!7,\!7,\!7)\in\RR^{\binom{8}{3}},\!\\[1mm]
    w_{4,\!8}=(&2,\!1,\!8,\!1,\!8,\!14,\!21,\!15,\!21,\!11,\!14,\!11,\!21,\!16,\!21,\!5,\!12,\!6,\!12,\!4,\!5,\!4,\!12,\!7,\!12,\!17,\!12,\\
               &17,\!19,\!19,\!19,\!17,\!12,\!17,\!19,\!14,\!21,\!15,\!21,\!11,\!14,\!11,\!21,\!16,\!21,\!21,\!15,\!21,\!21,\!19,\\[-1.5mm]
               &21,\!21,\!16,\!21,\!19,\!17,\!12,\!17,\!19,\!19,\!19,\!17,\!12,\!17,\!19,\!19,\!19,\!19,\!19,\!19)\in\RR^{\binom{8}{4}}.
  \end{align*}
  A corresponding reduced Gr\"obner basis of $\Grass(3,8)$ under the weighted monomial ordering with weight vector $w_{3,8}$ and lexicographical tiebreaker has $686$ elements of degrees ranging from $2$ to $6$, while the reduced Gr\"obner basis for $\Grass(4,8)$ has $1157$ elements of degrees ranging from $2$ to $8$.

  The Gr\"obner cone containing $w_{3,8}$ in its relative interior is of dimension $16$, generated by $9$ rays and a lineality space of dimension $8$, and the Gr\"obner cone with $w_{4,8}$ in its relative interior is of dimension $17$, generated by $10$ rays and a lineality space of dimension $8$.
  Hence both cones are maximal-dimensional in their respective tropical varieties and not simplicial.
\end{proof}

We conclude the section with some timings.

\begin{timings}
  Figure~\ref{fig:timings} compares three different algorithms for computing points on tropical varieties:
  \begin{description}
  \item[\textsc{gfan} 0.6.2] an experimental algorithm based on Chan's work on tropical curves \cite[Chapter 4]{Chan13}, and Jensen and Yu's work on stable intersections~\cite{JY16}.
  \item[\textsc{gfan} 0.5] \cite[Algorithm~9]{BJSST07}, a random traversal of the Gr\"obner fan while testing all rays for containment in the tropical variety. It can also be found in \textsc{Singular}, however that implementation is slower than \textsc{gfan}.
  \item[\textsc{Singular} 4.1.0] Algorithm~\ref{alg:tropicalStartingPoint}, as implemented in \texttt{tropicalNewton.lib}. As \textsc{gfan} additionally computes a corresponding reduced Gr\"obner basis, we also provide analogous timings in \textsc{Singular}.
  \end{description}
  We would like to stress that these timings merely serve as a comparison of the algorithms and not as a showcase of the computational reach of the two systems involved. For instance, points on tropical Grassmannians can also be computed via the tropical Stiefel map, see \cite[Proposition 12]{Herrmann2014} and~\cite{Fink2015}. In fact, $\mathcal G_{3,7}$ has been previously computed using \textsc{gfan 0.4}, which required 25 hours \cite[Theorem 2.2]{Herrmann2009}. Currently, \textsc{gfan 0.6.2} requires 65 minutes, while \textsc{Singular 4.1.0} requires 10 minutes.

  All computations were run on a machine with Intel E5-2643v3 (3.4 GHz) processors running Gentoo 4.4.6 and were aborted after exceeding 7 CPU days. See Definition~\ref{def:grassmannians} for the definitions of $\Det(k,n)$ and $\Grass(k,n)$. All examples are considered over $\CC\{\!\{t\}\!\}$.
  \begin{figure}[ht]
    \centering
    \begin{tabular}{|l|R{3cm}R{3cm}R{25mm}R{25mm}|}
      \hline
      &\textsc{Gfan 0.6.2}&\textsc{Gfan 0.5\phantom{.0}}&\multicolumn{2}{c|}{\textsc{Singular 4.1.0}}\\
      & & &$w\in\Trop(I)$&GB under $>_w$\\
      $\Det(2,5)$&1$\qquad$&1$\qquad$&1$\qquad\;$&1$\qquad\quad$\\
      $\Det(3,5)$&140$\qquad$&7$\qquad$&1$\qquad\;$&1$\qquad\quad$\\
      $\Det(2,6)$&5$\qquad$&1$\qquad$&1$\qquad\;$&1$\qquad\quad$\\
      $\Det(3,6)$&1800$\qquad$&900$\qquad$&8$\qquad\;$&1$\qquad\quad$\\
      $\Det(4,6)$&-$\qquad$&1100$\qquad$&41$\qquad\;$&1$\qquad\quad$\\
      $\Det(5,6)$&-$\qquad$&100$\qquad$&7$\qquad\;$&1$\qquad\quad$\\
      $\Grass(3,7)$&140$\qquad$&-$\qquad$&1$\qquad\;$&1$\qquad\quad$\\
      $\Grass(3,8)$&-$\qquad$&-$\qquad$&3$\qquad\;$&1$\qquad\quad$\\
      $\Grass(3,9)$&-$\qquad$&-$\qquad$&19$\qquad\;$&12$\qquad\quad$\\
      $\Grass(4,7)$&170$\qquad$&-$\qquad$&1$\qquad\;$&1$\qquad\quad$\\
      $\Grass(4,8)$&-$\qquad$&-$\qquad$&9$\qquad\;$&3$\qquad\quad$\\
      $\Grass(4,9)$&-$\qquad$&-$\qquad$&230$\qquad\;$&900$\qquad\quad$\\
      $\Grass(5,8)$&-$\qquad$&-$\qquad$&3$\qquad\;$&1$\qquad\quad$\\
      \hline
    \end{tabular}\vspace{-2mm}
    \caption{Timings in seconds, '-' were aborted after 7 CPU days.}
    \label{fig:timings}
  \end{figure}
\end{timings}

\begin{timings}\label{ex:fieldExtensionCont}
  Consider the computation of tropical points for the family of one-dimensional ideals generated by $F = \{f_1,\dotsc,f_n\} \subseteq \QQ[x_0,x_1,\dots,x_n]\subseteq\CC\{\!\{t\}\!\}[x_0,x_1,\dots,x_n]$ given by
  \[ f_k = x_k^2 - q_k x_0\Big( \sum_{i=1}^{k-1}x_i \Big) + q_k x_0^2, \text{ where }q_k\in\NN\text{ is the $k$-th prime}. \]
  While substituting $x_0\mapsto t$ directly yields triangular sets, we described in Example~\ref{ex:fieldExtension} how our Algorithm~\ref{alg:tropicalPoint0} struggles with them: It requires a degree $2^n$ field extension of $\QQ$, which results in a runtime of $8$ seconds for $n=13$, roughly doubling with each increase of $n$.

  However, for \cite[Algorithm~9]{BJSST07}, this family is completely trivial: As $F$ is already a reduced Gr\"obner basis for a suitable ordering, it is easy to verify that its ideal is one-dimensional and has a one-dimensional homogeneity space generated by $(1,1,\dots,1)\in\RR^{n+1}$. Hence \cite[Algorithm~9]{BJSST07} immediately obtains its tropical variety, which is equal to its homogeneity space. This shows in the runtime of both \textsc{gfan 0.5} and \textsc{gfan 0.6.2}, which terminate instantaneously for $n=13$ and whose runtimes remain under $1$ second for $n<120$.
  % The runtime of in Chan's tropical curve algorithm, see also Remark~\ref{rem:comparison}, depends heavily on which elimination ideals are computed: While the elimination of all variables but $x_0,x_1$ is instantaneous for all $n$, the elimination of all variables but $x_0,x_n$ fails to terminate within a day in \textsc{Singular 4.1.0} for $n=13$.
\end{timings}

\section{Computing tropical links}

In this section, we use Algorithm~\ref{alg:tropicalPoint0} to compute links of a tropical variety around its one-codimensional Gr\"obner polyhedra. This is done in two steps. First we intersect the link with a subspace to reduce it to a one-dimensional polyhedral fan. Afterwards, we intersect the fan with affine hypersurfaces to determine all its rays.
% We also apply the algorithm to show that $\mathcal G_{3,8}$ and $\mathcal G_{4,8}$ have codimension one links of higher valency.

\begin{definition}\label{def:tropicalLink}
  We refer to $\Trop(I)$ as a \emph{tropical link}, if it is a polyhedral fan and has a one-codimensional lineality space.
\end{definition}

\begin{remark}
  Let $u\in\Trop(I)$ sit in the relative interior of a one-codimensional Gr\"obner polyhedra. Then $\Trop(\initial_u(I)\otimes_{\mathfrak K} \mathfrak K\{\!\{t\}\!\})$ is a tropical link which describes $\Trop(I)$ locally around $u$, see Figure~\ref{fig:tropicalLink}. Its lineality space is the linear subspace spanned by the Gr\"obner polyhedra after moving $u$ to the origin.
  % The number of maximal Gr\"obner cones containing it is referred to as its \emph{valency}.
\end{remark}
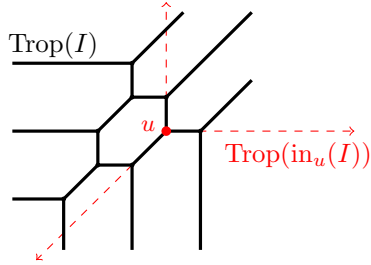
\begin{figure}[ht]
  \centering
  \begin{tikzpicture}[scale=0.9]
    \draw[red,dashed,->]
    (5.75,-0.25) -- (3.85,-2.15);
    \draw[red,dashed,->]
    (5.75,-0.25) -- node[below,xshift=0.5cm,font=\footnotesize] {$\Trop(\initial_u(I))$} (8.5,-0.25);
    \draw[red,dashed,->]
    (5.75,-0.25) -- (5.75,1.65);
    \draw[very thick] (4.25,-1.25) -- (4.25,-2)
    (4.25,-1.25) -- (3.5,-1.25)
    (4.25,-1.25) -- (4.75,-0.75)
    (4.75,-0.25) -- (4.75,-0.75)
    (4.75,-0.25) -- (3.5,-0.25)
    (4.75,-0.25) -- (5.25,0.25)
    (5.25,-0.75) -- (5.25,-2)
    (5.25,-0.75) -- (4.75,-0.75)
    (5.25,-0.75) -- (5.75,-0.25)
    (5.25,0.75) -- (5.25,0.25)
    (5.25,0.75) -- node[above,font=\footnotesize,xshift=-0.25cm,yshift=-0.1cm] {$\Trop(I)$} (3.5,0.75)
    (5.25,0.75) -- (6,1.5)
    (5.75,0.25) -- (5.75,-0.25)
    (5.75,0.25) -- (5.25,0.25)
    (5.75,0.25) -- (7,1.5)
    (6.25,-0.25) -- (6.25,-2)
    (6.25,-0.25) -- (5.75,-0.25)
    (6.25,-0.25) -- (7,0.5);
    \fill (4.25,-1.25) circle (1pt);
    \fill (4.75,-0.25) circle (1pt);
    \fill (4.75,-0.75) circle (1pt);
    \fill (5.25,-0.75) circle (1pt);
    \fill (5.75,0.25) circle (1pt);
    \fill[red] (5.75,-0.25) circle (2pt);
    \node[anchor=base east,font=\scriptsize,red] at (5.75,-0.25) {$u$};
    \fill (5.25,0.75) circle (1pt);
    \fill (5.25,0.25) circle (1pt);
    \fill (6.25,-0.25) circle (1pt);
  \end{tikzpicture}\vspace{-5mm}
  \caption{a tropical link of a tropical cubic curve}
  \label{fig:tropicalLink}
\end{figure}

The reduction to dimension zero relies on the following result on the intersection of tropical varieties by Osserman and Payne. From it, we can immediately write down our algorithm.

\begin{theorem}[{\cite[Theorem 1.1]{OssermanPayne13}}]% $\Trop(X)$ meets $\Trop(X')$ properly at $w$, i.e.
  Let $X$ and $X'$ be two affine subvarieties. If $\Trop(X)\cap\Trop(X')$ has codimension $\codim\Trop(X)+\codim\Trop(X')$ in a neighborhood of $w$, then $w$ is contained in $\Trop(X\cap X')$.
\end{theorem}

\begin{corollary}
  Let $\Trop(I)$ be a $d+1$-dimensional tropical link and $H$ its $d$-dimensional lineality space. Suppose $H\cap \Lin(e_{d+1},\dots,e_n) = \{0\}$. Then for any $z\in (K^\ast)^d$ we have
  \[ \Trop(I)\cap (\{\nu(z)\}\times\RR^{n-d}) = \Trop(I+\langle x_i-z_i\mid i=1,\dots,d\rangle), \]
  and $\Trop(I+\langle x_i-z_i\mid i=1,\dots,d\rangle)$ is a one-dimensional tropical link with lineality space $\{0\}$.
\end{corollary}

\begin{corollary}
  Let $\Trop(I)$ be a one-dimensional tropical link with lineality space $\{0\}$. Then for any $z\in K^\ast$ with $\nu(z)\neq 0$ we have
  \[ \Trop(I)\cap (\{\nu(z)\}\times\RR^{n-1}) = \Trop(I+\langle x_1-z\rangle), \]
  and $\Trop(I+\langle x_1-z\rangle)$ is either empty or zero-dimensional.
\end{corollary}

\begin{algorithm}[tropical link]\label{alg:tropicalLink}\
  \begin{algorithmic}[1]
    \REQUIRE{$I\unlhd K[x]$ such that $\Trop(I)$ is a $d+1$-dimensional tropical link with $d$-dimensional lineality space $H$.}
    \ENSURE{$W\subseteq \RR^n$ such that $\Trop(I)=\bigcup_{w\in W} w\cdot\RR_{\geq 0} + H$.}
    % \STATE Suppose $\dim(C_0(I))=d$, assume w.l.o.g.
    % \[C_0(I)\cap \Lin(e_{d+1},\dots,e_n)=\{0\}.\]\vspace{-0.5cm}
    \STATE Find $A\subseteq \{1,\dots,n\}$ such that $|A|=n-d$ and $H\cap\Lin(e_i\mid i\in A) = \{0\}$, say $A=\{d+1,\dots,n\}$.
    \STATE Let $J$ be the image of $I$ under the substitution map\vspace{-1.5mm}
    \[ \vspace{-1.5mm}K[x_1,\dots,x_n]\rightarrow K[x_d,\dots,x_n],\quad x_i\mapsto
    \begin{cases}
      1& \text{if }i<d,\\
      x_i& \text{else}.
    \end{cases} \]
    \FOR{$i=d,\dots,n$}
    \STATE Let $J_i^\pm$ be images of $J$ under the maps $x_i\mapsto p^{\pm 1}$ respectively.
    \IF{$J_i^\pm\neq\langle 1\rangle$}
    \STATE Compute $T_i^\pm=\Trop(J_i^\pm)$ using Algorithm~\ref{alg:tropicalPoint0}.\label{step:usingTropicalPoint0}
    \STATE Set\vspace{-1mm}
    \begin{align*}
      W_i^\pm:=&\{(0,\dots,0,w_d,\dots,w_{i-1},\pm 1,w_{i+1},\dots,w_n)\in\RR^n\mid\\
               &\qquad\qquad\qquad\qquad\quad(w_d,\dots,w_{i-1},w_{i+1},\dots,w_n)\in T_i^\pm\}.
    \end{align*}\vspace{-5mm}
    \STATE Scale elements of $W_i^\pm$ positively so that they are primitive in $\ZZ^n$.
    \ENDIF
    \ENDFOR
    \RETURN{$W:=\bigcup_{i=d}^n W_i^\pm$}
  \end{algorithmic}
\end{algorithm}

\begin{remark}[comparison with existing algorithms]\label{rem:comparison}
  The idea of computing tropical links by reducing the dimension is not new. Andrew Chan has designed an algorithm which computes tropical links via projection and reconstruction \cite[Section 4]{Chan13}, based on existing techniques developed by Hept and Theobald \cite{HT09}.

  In both algorithms, the polyhedral computations are timewise irrelevant compared to the polynomial computations, which contain three potential bottlenecks (assuming $K=\CC\{\!\{t\}\!\}$ and the use of \cite[Algorithm~4.7.8]{GreuelPfister} for the triangular decomposition necessary before applying Algorithm~\ref{alg:tropicalPoint0} to $J_i^\pm$):

  \begin{center}
    \begin{minipage}[t]{0.45\linewidth}
      \begin{description}[leftmargin=2mm]
      \item[{\cite[Algorithm 4.2.5, Step 1]{Chan13}}] computing elimination ideals
      \item[{\cite[Algorithm 4.2.14, Step 6]{Chan13}}] computing initial ideals
      \item[{\cite[Algorithm 4.2.14, Step 6]{Chan13}}] computing saturations
      \end{description}
    \end{minipage}%
    \hfill
    \begin{minipage}[t]{0.5\linewidth}
      \begin{description}[leftmargin=2mm]
      \item[({Algorithm~\ref{alg:tropicalLink}, Step~\ref{step:usingTropicalPoint0}})] \ \newline
        computing lexicographical Gr\"obner bases
      \item[({Algorithm~\ref{alg:tropicalLink}, Step~\ref{step:usingTropicalPoint0}})] \ \newline
        computing triangular decompositions
      \item[({Algorithm~\ref{alg:tropicalPoint0}, Step~\ref{step:rootComputation}})] \ \newline
        computing Newton-Puiseux expansions
      \end{description}
    \end{minipage}
  \end{center}

  Experiments suggest that, in both algorithms, the latter two bottlenecks are timewise insignificant compared to the first. In fact, for Algorithm~\ref{alg:tropicalLink}, constructing the triangular decomposition from a lexicographical Gr\"obner basis is polynomial~\cite[Section 7]{Lazard}, as is the construction of the Newton-Puiseux expansion~\cite{Chistov86}. Hence, the main bottleneck in both algorithms lies in the computation of Gr\"obner bases with respect to elimination orderings.

  % Both algorithms rely heavily on computations of Gr\"obner bases with respect to elimination orderings: \cite[Algorithm 4.2.5]{Chan13} requires the computation of several elimination ideals, while Algorithm~\ref{alg:tropicalLink} hides the computation of several lexicographical Gr\"obner bases, as applying Algorithm~\ref{alg:tropicalPoint0} to $J_i^\pm$ requires a triangular decomposition.
  However, the key difference is that these Gr\"obner basis computations in \cite[Algorithm~4.2.5]{Chan13} involve the one-dimensional input ideal $I$, whereas the ideals $J_i^\pm$ in Algorithm~\ref{alg:tropicalLink} are all zero-dimensional. For these ideals we not only have better complexity bounds \cite{Lazard83}, but also techniques such as \texttt{fglm} \cite{FGLM93}, which speed up our calculations drastically. For instance, in the following Example~\ref{ex:Gr49} and in \textsc{Singular 4.1.0}, a lexicographical Gr\"obner basis of $J_i^\pm$ required only 30 seconds of computation while an elimination ideal of $I$ required 25 minutes.

\end{remark}

\begin{example}[$\mathcal G_{4,9}$]\label{ex:Gr49}
  Let $K=\CC\{\!\{t\}\!\}$ and $I=\Grass(4,9)$. Its tropical variety $\mathcal G_{4, 9} \subseteq \RR^{126}$ is of dimension $21$ with a homogeneity space of dimension $9$.
  Using Algorithm~\ref{alg:tropicalPoint0}, one possible tropical point that lies in the interior of a maximal cone is\vspace{-1.5mm}
  \begin{align*}
    w\!:=\!(&1,\!1,\!3,\!4,\!4,\!7,\!8,\!9,\!9,\!9,\!10,\!10,\!9,\!5,\!10,\!10,\!10,\!4,\!4,\!12,\!13,\!13,\!13,\!13,\!5,\!17,\!18,\!18,\!18,\!18,\\
            &9,\!19,\!19,\!19,\!19,\!8,\!9,\!10,\!10,\!10,\!11,\!11,\!10,\!7,\!11,\!11,\!12,\!12,\!12,\!12,\!11,\!13,\!13,\!13,\!13,\!18,\\
            &\!19,\!19,\!19,\!19,\!10,\!20,\!20,\!20,\!20,\!20,\!20,\!20,\!20,\!20,\!10,\!10,\!10,\!10,\!12,\!13,\!13,\!13,\!13,\!11,\\
            &17,\!18,\!18,\!18,\!18,\!11,\!19,\!19,\!19,\!19,\!18,\!19,\!19,\!19,\!19,\!10,\!20,\!20,\!20,\!20,\!20,\!20,\!20,\!20,\\[-1.5mm]
            &\!20,\!18,\!19,\!19,\!19,\!19,\!11,\!20,\!20,\!20,\!20,\!20,\!20,\!20,\!20,\!20,\!20,\!20,\!20,\!20,\!20,\!20)\in\RR^{\binom{9}{4}}.\vspace{-1.5mm}
  \end{align*}
  The reduced Gr\"obner basis of the initial ideal under $w$ with respect to the reverse lexicographical ordering consists of $5543$ binomials with degrees ranging from $2$ to $7$.
  The Gr\"obner cone $C_w(I)$ is simplicial with its $12$ facets.
  Figure~\ref{fig:GBcomparison} shows some data on the reduced Gr\"obner bases of the saturated initial ideals under weight vectors on the facets of $C_w(I)$. The rows represent binomials, trinomials and quadrinomials respectively and the columns represent degrees $2$ to $7$, i.e. the entry in row $i$ and column $j$ is the number of Gr\"obner basis elements with $i+1$ monomials and of degree $j+1$.

  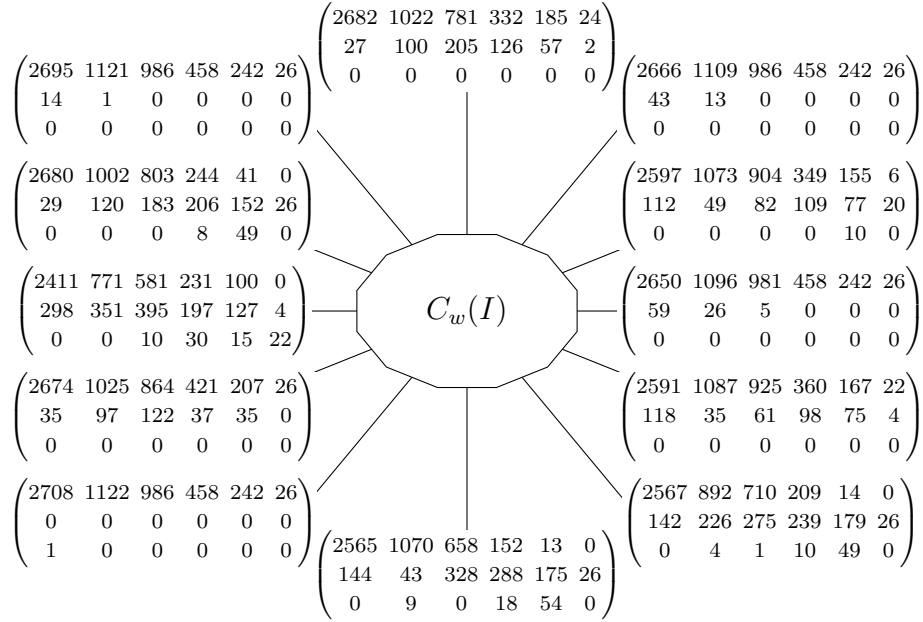
\begin{figure}[ht]
    \centering
    \begin{tikzpicture}[yscale=0.7]
      \draw (0,0) -- (0:5cm)
      (0,0) -- (30:5cm)
      (0,0) -- (60:4.5cm)
      (0,0) -- (90:5cm)
      (0,0) -- (120:4.5cm)
      (0,0) -- (150:5cm)
      (0,0) -- (180:5cm)
      (0,0) -- (210:5cm)
      (0,0) -- (240:4.5cm)
      (0,0) -- (270:5cm)
      (0,0) -- (300:4.5cm)
      (0,0) -- (330:5cm);
      \node[font=\tiny,fill=white,inner sep=0pt,outer sep=1pt] at (4,0) {$\facetstatsTwo$};
      \node[font=\tiny,fill=white,inner sep=0pt,outer sep=1pt] at (4,2) {$\facetstatsOne$};
      \node[font=\tiny,fill=white,inner sep=0pt,outer sep=1pt] at (4,4) {$\facetstatsTen$};
      \node[font=\tiny,fill=white,inner sep=0pt,outer sep=1pt] at (0,5) {$\facetstatsFive$};
      \node[font=\tiny,fill=white,inner sep=0pt,outer sep=1pt] at (-4,4) {$\facetstatsThree$};
      \node[font=\tiny,fill=white,inner sep=0pt,outer sep=1pt] at (-4,2) {$\facetstatsSeven$};
      \node[font=\tiny,fill=white,inner sep=0pt,outer sep=1pt] at (-4,0) {$\facetstatsSix$};
      \node[font=\tiny,fill=white,inner sep=0pt,outer sep=1pt] at (-4,-2) {$\facetstatsEight$};
      \node[font=\tiny,fill=white,inner sep=0pt,outer sep=1pt] at (-4,-4) {$\facetstatsNine$};
      \node[font=\tiny,fill=white,inner sep=0pt,outer sep=1pt] at (0,-5) {$\facetstatsFour$};
      \node[font=\tiny,fill=white,inner sep=0pt,outer sep=1pt] at (4,-4) {$\facetstatsTwelve$};
      \node[font=\tiny,fill=white,inner sep=0pt,outer sep=1pt] at (4,-2) {$\facetstatsEleven$};
      % \node[font=\footnotesize,fill=white] at (0:1.25cm)
      % {$\facetvalencyTwo$};
      % \node[font=\footnotesize,fill=white] at (30:1.25cm)
      % {$\facetvalencyOne$};
      % \node[font=\footnotesize,fill=white] at (60:1.25cm)
      % {$\facetvalencyTen$};
      % \node[font=\footnotesize,fill=white] at (90:1.25cm)
      % {$\facetvalencyFive$};
      % \node[font=\footnotesize,fill=white] at (120:1.25cm)
      % {$\facetvalencyThree$};
      % \node[font=\footnotesize,fill=white] at (150:1.25cm)
      % {$\facetvalencySeven$};
      % \node[font=\footnotesize,fill=white] at (180:1.25cm)
      % {$\facetvalencySix$};
      % \node[font=\footnotesize,fill=white] at (210:1.25cm)
      % {$\facetvalencyEight$};
      % \node[font=\footnotesize,fill=white] at (240:1.25cm)
      % {$\facetvalencyNine$};
      % \node[font=\footnotesize,fill=white] at (270:1.25cm)
      % {$\facetvalencyFour$};
      % \node[font=\footnotesize,fill=white] at (300:1.25cm)
      % {$\facetvalencyTwelve$};
      % \node[font=\footnotesize,fill=white] at (330:1.25cm)
      % {$\facetvalencyEleven$};
      \draw[fill=white]
      (15:1.5cm)
      -- (45:1.5cm)
      -- (75:1.5cm)
      -- (105:1.5cm)
      -- (135:1.5cm)
      -- (165:1.5cm)
      -- (195:1.5cm)
      -- (225:1.5cm)
      -- (255:1.5cm)
      -- (285:1.5cm)
      -- (315:1.5cm)
      -- (345:1.5cm)
      -- cycle;
      \node at (0,0) {$C_w(I)$};
    \end{tikzpicture}\vspace{-5mm}
    \caption{reduced Gr\"obner bases around a maximal cone in $\mathcal G_{4,9}$}
    \label{fig:GBcomparison}
  \end{figure}

  The computation of the $12$ tropical links using Algorithm~\ref{alg:tropicalLink} took $7$ minutes, while all attempts to compute any of the $12$ tropical prevarieties failed to terminate within an hour, even using the newly developed techniques by Jensen, Sommars and Verschelde \cite{JSV17}.  Similarly, computing any of the elimination ideals necessary in \cite[Algorithm~4.2.5]{Chan13} required 25 minutes in \textsc{Singular}. All tropical links are $3$-valent, i.e. each facets is adjacent to exactly three maximal cones in the tropical variety.
\end{example}

% Using Algorithm~\ref{alg:tropicalLink}, we conclude:

% \begin{theorem}\label{thm:valency}
%   The tropical Grassmannian $\mathcal G_{d,n}$ have links of valency bigger than three for $d=3,4$ and $n=8$.
% \end{theorem}
% \begin{proof}
%   In the proof of Theorem~\ref{thm:simpliciality}, we have constructed two maximal cones which were not simplicial.

%   The maximal cone in $\mathcal G_{3,8}$ has $9$ facets in total. Using Algorithm~\ref{alg:tropicalLink}, we see that five have valency $4$, two have valency $8$ and there two links each with valency $16$ and $24$.

%   The maximal cone in $\mathcal G_{3,8}$ has $10$ facets in total. Using Algorithm~\ref{alg:tropicalLink}, we see that two have valency $4$, two have valency $6$, two have valency $16$ and there are four links each with valency $8$, $16$, $20$ and $46$.
% \end{proof}

\renewcommand{\emph}[1]{\textit{#1}}
\renewcommand*{\bibfont}{\small}
\printbibliography

\end{document}